\documentclass[12pt]{amsart}

\usepackage{amsmath, amssymb, amsthm, amsfonts, amscd}

\input xy
\xyoption{all}

\usepackage{hyperref}
\usepackage{graphicx}
\usepackage{color}
\usepackage{tikz-cd}

\makeatletter
\def\@tocline#1#2#3#4#5#6#7{\relax
  \ifnum #1>\c@tocdepth 
  \else
    \par \addpenalty\@secpenalty\addvspace{#2}%
    \begingroup \hyphenpenalty\@M
    \@ifempty{#4}{%
      \@tempdima\csname r@tocindent\number#1\endcsname\relax
    }{%
      \@tempdima#4\relax
    }%
    \parindent\z@ \leftskip#3\relax \advance\leftskip\@tempdima\relax
    \rightskip\@pnumwidth plus4em \parfillskip-\@pnumwidth
    #5\leavevmode\hskip-\@tempdima
      \ifcase #1
       \or\or \hskip 3em \or \hskip 6em \else \hskip 9em \fi%
      #6\nobreak\relax
    \hfill\hbox to\@pnumwidth{\@tocpagenum{#7}}\par
    \nobreak
    \endgroup
  \fi}
\makeatother

\numberwithin{equation}{section}

\theoremstyle{plain}
\newtheorem{theorem}[equation]{Theorem}
\newtheorem{lemma}[equation]{Lemma}

\newtheorem{proposition}[equation]{Proposition}
\newtheorem{corollary}[equation]{Corollary}

\theoremstyle{definition}
\newtheorem{definition}[equation]{Definition}

\newtheorem{remark}[equation]{Remark}

\newtheorem{example}[equation]{Example}

\newtheorem{assumption}[equation]{Assumption}

\usepackage[margin=1in,marginparwidth=0.8in, marginparsep=0.1in]{geometry}

\def\AA{\mathbb{A}}

\def\CC{\mathbb{C}}

\def\GG{\mathbb{G}}

\def\LL{\mathbb{L}}

\def\PP{\mathbb{P}}

\def\RR{\mathbb{R}}
\def\R{\mathbb{R}}

\def\VV{\mathbb{V}}

\def\ZZ{\mathbb{Z}}

\def\cC{\mathcal{C}}
\def\cD{\mathcal{D}}

\def\cL{\mathcal{L}}
\def\cM{\mathcal{M}}
\def\cN{\mathcal{N}}

\def\cS{\mathcal{S}}

\def\cZ{\mathcal{Z}}

\def\bL{\mathbf{L}}

\def\bT{\mathbf{T}}
\def\bU{\mathbf{U}}

\def\bW{\mathbf{W}}

\def\rT{\mathrm{T}}

\newcommand{\Coh}{\textup{Coh}}

\newcommand\Fuk{\textup{Fuk}}

\newcommand\Mod{\textup{Mod}}

\newcommand\Spec{\textup{Spec}}

\newcommand\Hom{\textup{Hom}}

\newcommand\nc{\newcommand}
\nc\on{\operatorname}
\nc\ol{\overline}
\nc\ul{\underline}

\nc\oo{\infty}

\nc\Cone{\mathit{Cone}}
\nc\ssupp{\mathit{ss}}
\nc\risom{\stackrel{\sim}{\to}}
\nc\Sh{\mathit{Sh}}
\nc\un{\diamondsuit}
\nc\orient{\mathit{or}}
\nc\sing{\mathit{sing}}
\nc\MF{\on{MF}}
\nc\Log{\on{Log}}
\nc\Arg{\on{Arg}}
\nc\inthom{\mathit{Hom}}
\nc\colim{\on{colim}}

\nc\wdv{W_{\Delta^\vee}}
\nc\trt{\widetilde{\on{trop}}}
\nc\dtmir{\partial\bT^{mir}}
\nc\ttmir{\partial\bT^{mir}_{loc}}
\nc\FS{\on{FS}}
\nc\mshw{\mu sh^c}

\nc\conv{\mathit{conv}}
\nc\trad{\mathit{inf}}
\nc\wrap{\mathit{wr}}

\nc\lvl{X_0}
\nc\lcs{X_0^\vee}
\nc\crr{\color{red}}
\nc\tcl{\widetilde{\cL}}
\nc\LGr{\on{LGr}}
\nc\Conv{\on{Conv}}

\DeclareMathOperator{\Exit}{Exit}

\newcommand{\boing}{\mbox{$ \bigcirc \!\! \bullet$}}

\newcommand{\mapsfrom}{\mathrel{\reflectbox{\ensuremath{\mapsto}}}}

\begin{document}


\title[Homological mirror symmetry at large volume]{Homological mirror symmetry at large volume}
\author{Benjamin Gammage}
\address{Department of Mathematics, Harvard University, Cambridge, MA 02138, USA}
\email{gammage@math.harvard.edu}

\author{Vivek Shende}
\address{Centre for Quantum Mathematics, SDU, Campusvej 55, 5230 Odense M, Denmark}
\email{vivek.vijay.shende@gmail.com}

\begin{abstract}
	A typical large complex-structure limit for mirror symmetry consists of toric varieties glued
	to each other along their toric boundaries.   Here we construct the mirror large volume limit
	space as a Weinstein symplectic manifold.  We prove homological mirror symmetry: the category
	of coherent sheaves on the first space is equivalent to the Fukaya category of the second. 
	
	Our equivalence intertwines the Viterbo restriction maps for a generalized pair-of-pants cover of the symplectic manifold 
	with the restriction of coherent sheaves for a certain affine cover of the algebraic variety. 
	We deduce a posteriori a local-to-global principle conjectured by Seidel
	--- certain diagrams of Viterbo restrictions are cartesian --- by passing Zariski descent through our mirror symmetry result.
\end{abstract}

\maketitle


\tableofcontents


\newpage

\section{Introduction}

Homological mirror symmetry concerns the existence of isomorphisms 
$$\mathrm{Fuk}(X) = \mathrm{Coh}(Y)$$
between the Fukaya category of a symplectic manifold and the category of coherent sheaves on an algebraic variety \cite{Kontsevich-ICM}. 
It serves as an underlying explanation for various matchings of holomorphic curve invariants
of $X$ with Hodge-theoretic invariants of $Y$, which are recovered from the categorical statement by studying
Hochschild homology and related structures \cite{Costello, KKP, Ganatra-Perutz-Sheridan, Caldararu-Tu}.

The model case in which mirror symmetry may be expected to hold is when $X$ and $Y$ admit dual ``SYZ'' Lagrangian torus fibrations over a 
common base \cite{SYZ}.  Indeed, in this case a local system on a torus fiber of $X$ yields both an object of $\Fuk(X)$ and also 
a point on the dual torus (hence a skyscraper sheaf in $\Coh(Y)$).  However, for general $X, Y$ one will need fibrations
with singularities, which are hard to construct, hard to dualize, and hard to study \cite{Kontsevich-Soibelman}. 

Mirror symmetry is best understood when the algebraic variety is infinitesimally near in moduli to a ``large complex structure limit.''  
For example, for a hypersurface in projective space, this means being close to a union of hyperplanes.  
A natural and very general setting for studying such phenomena is that of toric degenerations \cite{GS-log1, GS-log2}.  In the present article
we will be interested for the most part solely in the central fiber of such a degeneration: our $Y$ will be a union of toric
varieties glued along their toric boundaries.  Note that as $Y$ has singularities, 
the derived Hom spaces between coherent sheaves may be unbounded even if $Y$ is compact.  

A choice of ample line bundle on $Y$ gives moment maps to its toric components, and hence a 
fibration  $\sigma: Y \to \Psi$
over some space $\Psi$ glued from their moment polytopes: this fibration should be understood as the SYZ fibration. 
Traditionally, the base $\Psi$ is a topological manifold (but see Example \ref{ex:3A1}, in which we consider the case of three lines meeting at a point).
Note that $\Psi$ carries naturally an integral affine structure on each face of these polytopes --- that is, along the tangent directions to the face, but not in their normal directions.\footnote{Gross and Siebert argue
in \cite{GS-log1, GS-log2} that appropriate deformations of  $Y$ to a smooth variety  correspond to extensions of the integral affine structure 
over the complement of a codimension-2 discriminant in $\Psi$.  As we are interested solely in the central fiber, this will play no role here.}  
The space $\Psi$ also carries a sheaf of lattices $\mathrm{R}^1 \sigma_* \ZZ$.  
The SYZ picture suggests that there should be an embedding  $\mathrm{R}^1 \sigma_* \ZZ \hookrightarrow \rT^*\Psi$ and 
$X = \rT^*\Psi / \mathrm{R}^1 \sigma_* \ZZ$. 
(Note that this will have noncompact fibers over the positive-codimensional strata of $\Psi$.) 

However, we will prefer a dual picture.  
As we are only interested in the complex (rather than K\"ahler) geometry of $Y$, it is not natural to require the  toric varieties in $Y$ be polarized or even projective. 
Instead of a moment polytope, we associate to a toric variety its fan, the collection of rational polyhedral
cones of toric cocharacters with the same limit point.  We glue together these fans into a  stratified space $\Phi$, whose 
$k$-dimensional strata correspond to $k$-dimensional cones in fans of components of $Y$. 
When the SYZ base $\Psi$ is well-defined, the space $\Phi$ is dual to it in a sense that can be made precise in terms of the discrete Legendre transform; at the topological level, this duality entails an anti-isomorphism between the stratification posets of $\Psi$ and $\Phi$. In duality to the situation for $\Psi,$ each $k$-stratum in $\Phi$ has an integral affine structure {\em only in its  normal directions}. 

In Section \ref{sec: fanifolds}, we make precise the notion of a stratified manifold where the normal geometry to each stratum is a fan. 
We will show that these {\em fanifolds} provide the organizing topological and discrete data for homological mirror symmetry at large volume.  

In Section \ref{sec:lcs}, we explain how a fanifold encodes a gluing of toric varieties and organizes functorialities and descent
among their categories of coherent sheaves.  More precisely, we describe a functor $\bT$ from fanifolds to algebraic spaces
and define $Y = \bT(\Phi)$.  

As we have mentioned, the SYZ picture indicates that $X$ should be noncompact.  As $Y$ is singular, 
there are necessarily infinite dimensional (derived) Hom spaces in $\Coh(Y)$, which will correspond to the infinite-dimensional Hom spaces in the {\em wrapped} 
Fukaya category of $X$. In order to define the wrapped category, we require not only the data of the symplectic manifold $X$,
but the additional choice of a conical primitive at least in the complement of a compact set.  In fact, $X$ will be exact:  
in Section \ref{sec:lvl} we explain how to construct for any $\Phi$ a Weinstein manifold $X=\bW(\Phi)$.  

Our construction of $\bW(\Phi)$ is guided by the idea of gluing together mirror symmetry for toric varieties.  
From \cite{FLTZ2, Ku} we know that the mirror category to a toric variety
with fan $\Sigma = \bigcup \sigma$ is controlled by the conic Lagrangian 
$\LL(\Sigma) := \bigcup \sigma^\perp \times \sigma$ in the conormal bundle to a torus.
(Moreover, when $\Sigma$ is simplicial, it is shown in \cite{GS17, Zhou-skel} that $\LL(\Sigma)$ is the relative skeleton of 
the Liouville sector associated to the expected Hori-Vafa superpotential.)  
In that case, the projection to the cotangent fiber
gives a stratified map $\LL(\Sigma) \to \Sigma$; the guiding principle behind our construction of $\bW(\Phi)$ is 
that its skeleton should have an analogous map $\LL(\Phi) \to \Phi$.  

It remains to prove homological mirror symmetry, which we accomplish as follows.
By \cite{GPS3} we may calculate $\bW(\Phi)$ in terms of the sheaf of categories of microsheaves (as defined in \cite{shende-microlocal, NS20}) over $\LL(\Phi)$; 
pushing this forward, we obtain a constructible sheaf of categories over $\Phi$.  
Meanwhile by \cite[Chap. 8.A, Thm. A.1.2]{GR2}\footnote{The importance 
of this result for mirror symmetry was first recognized in \cite{Nadler-wrapped}.}  we may calculate 
the coherent sheaves over $\bT(\Phi)$ in terms of a limit of categories of coherent sheaves on toric varieties; 
in other words, $\Coh(\bT(\Phi))$ is computed as global sections of a second sheaf of categories over $\Phi$.  The main calculation of \cite{GS17} shows that these 
two sheaves of categories are locally equivalent. 
To establish a global equivalence, we will need a way of calculating microsheaf categories from local pieces; 
in Section \ref{sec: microsheaves}, we explain how the fact 
that $\bW(\Phi)$ has a cover by cotangent bundles in which $\LL(\Phi)$ is conical
makes this calculation possible.

We state our mirror theorems in Section \ref{sec:hms} and prove them there under the hypothesis that all fans are smooth; 
the following Section \ref{sec: singular} removes this hypothesis.  
All the results of the article in fact hold in the setting of stacky fans; for expository reasons we restrict ourselves to ordinary fans 
until Section \ref{sec: singular}, where we explain the minor adjustments required to work in the stacky case.

A key feature of our mirror symmetry result is that it interwines certain Viterbo restriction functors of Fukaya categories with 
restriction of coherent sheaves to Zariski-open subsets.  This is of paramount importance to the problem of
deforming our mirror symmetry result to one for smooth compact manifolds, as we explain in Section \ref{sec: epilogue}.

\vspace{2mm} 

{\bf Categorical notions and notations.}  For foundations on DG categories, we refer to the development in \cite{GR1}.  

We write ${}^* \mathrm{DG}$ for the category of colimit-complete dg categories and colimit-preserving functors; the notation reminds us
that all morphisms are left adjoints when viewed as morphisms of dg categories (though their adjoints will not in general live in 
${}^* \mathrm{DG}$).  Similarly, we write $\mathrm{DG}^*$ for the category of colimit-complete dg categories and limit-preserving functors. 
Note that taking adjoints gives an identification $({}^* \mathrm{DG})^{op} = \mathrm{DG}^*$, so that
 colimits in ${}^* \mathrm{DG}$ can be computed as limits in $\mathrm{DG}^*$, which are simply limits in the category of categories. 

We write $\mathrm{dg}$ for the category of small dg categories.  Taking Ind-objects (or equivalently, passing to module categories) gives a full embedding 
$\mathrm{dg} \to {}^* {}^* \mathrm{DG}$ into the category of colimit-complete dg categories and 
functors which preserve colimits and compact objects (equivalently, functors which are left adjoints of left adjoints).  The image of this embedding is the compactly generated categories. 

We always use the least decorated name for the presentable variant of a DG category.  Thus we write $\Coh$ for what 
is termed $\mathrm{IndCoh}$ in \cite{GR1, GR2}, and $\Fuk$ for what would elsewhere be called the category of modules
over the Fukaya category. 

\vspace{2mm} 

{\bf Acknowledgements.} BG is supported by an NSF postdoctoral fellowship, DMS-2001897. 
VS is partially supported by NSF CAREER DMS-1654545.

\section{Fanifolds}  \label{sec: fanifolds}
In this section, we make precise the notion of a stratified manifold for which the geometry normal to each stratum is equipped with 
the structure of a fan.  

Let us fix some notation for stratified spaces.  We will consider only spaces which are stratified by with finitely many strata and which are conical in the complement
of a compact set. For such a stratified space $\cS$, we write $\overline{\cS}$ for the natural compactification.

For a stratified space $\cS$, we write $\Exit(\cS)$ for the exit path category of $\cS$. 
These serve to organize constructible sheaves, which are equivalent to functors from $\Exit(\cS)$, and constructible cosheaves,
which are equivalent to functors from $\Exit(\cS)^{op}$. For a stratum $F,$ we write $\Exit_F(\cS)$ for the category of exit paths starting at $F$ contained inside a sufficiently small neighborhood $Nbd(F)$.

Consider a stratified space $\cS$ which is given as a germ of a closed subset in a manifold and whose strata are smooth submanifolds.  
We will express properties of $\cS$ in terms of a choice of 
ambient manifold $\cM$, although they will only depend on the germ of $\cS.$
Fix a stratum $F$ of $\cS$. 
Taking deformation to the normal cone, we obtain a stratification of the normal cone $C_F \cS \subset \rT_F \cM$.  
We say that $\cS$ is {\em smoothly normally conical} if some choice of tubular neighborhood 
$\rT_F \cM \to \cM$ induces locally near $F$ a stratified diffeomorphism $C_F \cS \to \cS$, inducing the identity $C_F \cS \to C_F \cS$
upon deformation to the normal cone.  

\begin{example}
A typical example of a smoothly normally conical stratification is a polytopal decomposition of a vector space; 
a typical nonexample is the cusp $y^2 = x^3$. 
\end{example}

\begin{definition}We denote by $\mathrm{Fan}^{\twoheadrightarrow}$ a category whose objects are pairs $(M, \Sigma),$ where $M$ is a lattice and $\Sigma$ is stratified by 
finitely many rational polyhedral cones in $M \otimes \R$.  A map $(M, \Sigma) \to (M', \Sigma')$ in $\mathrm{Fan}^{\twoheadrightarrow}$ is 
the data of a cone $\sigma \in \Sigma$ and an isomorphism $M / \mathrm{Span}(\sigma) \cong M'$, such that $\Sigma'$ consists of the images
of cones in $\Sigma$ whose closures contain $\sigma$; we may therefore write $\Sigma'   = \Sigma / \sigma$.
\end{definition}
\begin{remark}
There are other interesting notions of morphisms of fans, which is why we use the decoration $\twoheadrightarrow$. 
\end{remark}

We may view a fan $\Sigma$ as a space stratified by its cones; 
evidently, it is smoothly normally conical.  We have a natural identification of posets: 
\begin{eqnarray*} 
\Exit(\Sigma) & \cong & \mathrm{Fan}^{\twoheadrightarrow}_{\Sigma / } \\
\sigma & \mapsto & [\Sigma \mapsto \Sigma/\sigma]
\end{eqnarray*}
In addition, the normal geometry to $\sigma$ is that of
the fan $\Sigma / \sigma$.   We wish to study stratified spaces whose normal geometry
has this local model.

\begin{definition}\label{defn:fanifold}
  A {\em fanifold} is a smoothly normally conical stratified space $\Phi \subset \cM$, equipped with the following data: 
   \begin{itemize}
   \item A functor $\Exit(\Phi) \to \mathrm{Fan}^{\twoheadrightarrow}$.  The value of this functor on a stratum $F$ determines a (local system of) lattice $M_F$ and  
   rational polyhedral fan $\Sigma_F \subset  M_F \otimes \RR$.     
   \item For each stratum $F \subset \cS$, an isomorphism of the normal bundle $\phi: \rT_F \cM \cong M_F \otimes \RR$ carrying the induced stratification
      on the normal cone $C_F \Phi$ to the standard stratification induced by the fan.  
   \end{itemize} 
   These data are required to satisfy the following compatibility:  
   \begin{itemize}
     \item Given a stratum $F'$ of the induced stratification on $Nbd(F)$ (equivalently, given a stratum $F'$ of the normal cone $C_F \Phi$), the above trivialization
   gives an associated cone $\sigma_{F'} \subset \Sigma_F$.  The tangent bundle to $F'$ naturally extends over $F$, and 
   we get an induced map on normal bundles $\rT_F \cM \to \rT_{F'} \cM|_F$, which is the quotient by the span of $\sigma_{F'}$.   
   We ask that this map is intertwined by $\phi$ with the corresponding map on lattices $M_F \to M_F'$.   (In case 
   $\Sigma_F$ spans $M_F$, this is automatic.) 
   \end{itemize}
\end{definition}

Note that we may regard a manifold as a fanifold: the normal to the unique stratum is simply zero, and we equip it with the trivial fan.  The product of a manifold
and fanifold is naturally a fanifold; in fact, since the product of fans is a fan, the product of fanifolds is a fanifold in a natural way.  

However, we impose a restriction
on the fanifolds we will study here:
\begin{assumption}
  Throughout this paper, we require that all strata in the fanifolds we consider are contractible, so that the local systems of lattices appearing in Definition~\ref{defn:fanifold} are trivial.
  In this case, it follows from the definition that $\Exit_F(\Phi)$
  is equivalent to the poset $\Exit(\Sigma_F)$.   
\end{assumption}

\begin{example} \label{ex: trivial} 
Consider a single point $p$ in some ambient $n$-manifold $\cM$. The point $p$ carries a natural fanifold structure,
where we equip the normal bundle to the point with the trivial fan of cones, whose only cone is $\{0\}$. (This fan is familiar in toric geometry as the fan of cones describing the toric variety $\GG_m^n.$)
\end{example}

\begin{example} \label{ex: necklace} 
Consider a circle stratified by $r$ points and $r$ intervals, for any $r \ge 1$.  This space
acquires a fanifold structure from the evident identification of the normal geometry at each point 
with the fan in $\R$ consisting of the origin and both rays.  For $r = 1$ we denote this fanifold as $\boing$. 
\end{example} 

\begin{example} \label{ex: fan} 
Let $\Sigma$ be a fan of cones in $\RR^n$, and view $\Sigma$ as stratified by its cones.  Then equipping the stratum $\sigma$ with normal fan $\Sigma/\sigma$ determines a fanifold structure on $\Sigma$. 
\end{example}

\begin{example}
Consider $\RR^n$ equipped with an integral polyhedral decomposition in the sense of \cite[Definition 2.3]{GS-intro}. Each $k$-face is contained in a $k$-dimensional affine subspace of $\RR^n$, so that its normal directions are the quotient of $\RR^n$ by its affine span, 
and the faces incident on it determine a fan of cones in this quotient, giving $\RR^n$ a fanifold structure. 
\end{example}

\begin{example}
If $\Phi \subset \cM$ is a fanifold, and $\cN \subset \cM$ is a submanifold transverse to all strata in $\Phi$, then 
$\Phi \cap \cN \subset \cN$ inherits a fanifold structure in an evident way.  
\end{example}

\begin{example}\label{ex:bdry-normalfan}
  If $\Sigma$ is a fan of cones in $\RR^{n+1}$, and we write $S^n:=\{\sum x_i^2=1\}\subset \RR^{n+1}$ for the standard embedding of the $n$-sphere, then $\Sigma \cap S^n \subset S^n$ carries a natural fanifold structure. 
\end{example}

Recall we assume that our stratified spaces are conical at infinity.  Given a fanifold $\Phi$, we write $\partial_\infty \Phi$ for the ideal boundary --- 
i.e. the stratified space for which, in the complement of a compact set, $\Phi \approx \RR_{>0} \times \partial_\infty \Phi $.   The ideal boundary 
$\partial_\infty \Phi$ may be identified with a hypersurface in $\Phi$, and as such carries a natural fanifold structure.  (One might also formulate
a notion of fanifold with boundary, but we will avoid doing so here.)  

Fanifolds may be glued along common subsets of their ideal boundary.  That is, given $n$-dimensional fanifolds $\Phi, \Phi'$ and an $(n-1)$-dimensional fanifold $U$,
along with open embeddings (with boundary transverse to all strata) $U \hookrightarrow \partial_\infty \Phi$ and $U \hookrightarrow \partial_\infty \Phi'$,  
there is a natural gluing $\Phi \#_{U} \Phi'$.  

\begin{remark} \label{rem: handle attachment fanifold} 
One can always think of a fanifold $\Phi$ as being constructed by iterative ``handle attachments,'' in the following way.  At a 0-stratum $P$, 
by definition, the local geometry is that of some fan $\Sigma_P$.  So begin with a disjoint union of fans $\Phi_0 := \coprod \Sigma_P$ corresponding to the 0-strata, 
equipped with the canonical fanifold structure.  For each 1-stratum $I$, there will be some transverse
fan $\Sigma_I$.  The ideal boundary $\partial_\infty I$ will have some subset $\partial_{in} I$ which is in the direction of the interior of $\Phi$, and
we may use it to perform a gluing $\Phi_0 \#_{\Sigma_I \times \partial_{in} I} (\Sigma_I \times I)$.  
(Note that $\Phi$ may not have 0-strata, in which case $\partial_{in} I$ will be empty, as will $\Phi_0,$ and this gluing will be trivially equal to $\Sigma_I\times I.$)
Doing this for all 1-strata yields a 1-dimensional fanifold $\Phi_1$, to which we 
then attach handles $S \times \Sigma_S$ for each  2-stratum $S$, and so on.   

Note that not all handle attachments change the geometry.  For example, when $\Phi = \Sigma$ is already a fan, after beginning with $\Phi_0 = \Sigma$, 
we still attach 1-handles, 2-handles, etc., for all the remaining cones of the fan, but each of these handle attachments acts trivially on the fanifold
(just as a connect sum with a ball acts trivially on a manifold).  

In Section \ref{sec:lvl-defn}, we will lift this procedure to a construction of 
symplectic manifolds, essentially by replacing $S$ by $T^*S$ and $\Sigma$
by the FLTZ skeleton associated to $\Sigma$. 
\end{remark}

\section{B-model}\label{sec:lcs}

A lattice $M$ and a fan $\Sigma \subset M \otimes \RR$ classically determine a toric 
variety $\bT(\Sigma),$ constructed
by gluing the affine varieties $\mathrm{Spec}\, k[\sigma]$ via the evident inclusions.  A cone $\sigma \in \Sigma$ also
determines a toric orbit $O(\sigma)$ in $\bT(\Sigma)$ whose closure 
$\overline{O(\sigma)}$ is itself a toric variety, canonically isomorphic to  $\bT(\Sigma/\sigma)$.   

That is, there is a functor 
\begin{equation}\label{eq:toric-functor}
\bT: (\mathrm{Fan}^{\twoheadrightarrow})^{op} \to \mathrm{Schemes}
\end{equation}
carrying all morphisms to closed embeddings.  
In this section, we study the extension of this functor from fans to fanifolds. 

\subsection{Schemes from fanifolds} 
A fanifold $\Phi$ includes the data of a functor $\Exit(\Phi) \to \mathrm{Fan}^{\twoheadrightarrow}$.  We denote the composition of this functor with \eqref{eq:toric-functor}
also by $$\bT: \Exit(\Phi)^{op} \to \mathrm{Schemes}.$$ 
A functor with domain $\Exit(\Phi)^{op}$ would
define a cosheaf if it were valued in a cocomplete
category, in which case its global sections would be computed by the colimit
$$\bT(\Phi) :=  \varinjlim_{\Exit(\Phi)^{op}}\bT({F}).$$

Unfortunately, the category $\mathrm{Schemes}$ is not cocomplete; in fact, it is not even guaranteed to contain colimits of pushout diagrams.
Although a pushout diagram of affine schemes
$A \leftarrow B \rightarrow C$ certainly admits a pushout in the category of affine schemes --- namely the spectrum of the limit
of the corresponding rings --- this need not agree with the pushout in the category of schemes, even when both exist.
(Consider for instance the diagram $\Spec\, k[t] \leftarrow \Spec\,k[t,t^{-1}] \rightarrow \Spec\, k[t^{-1}]$.)
However, if at least one of the morphisms is 
a closed immersion, then the pushout of affine schemes is the same as the pushout of schemes \cite[Theorem 3.4]{Schwede}. 

Moreover, for arbitrary diagrams of schemes $A \leftarrow B \rightarrow C$ in which both morphisms are closed immersions, 
the pushout exists \cite[Corollary 3.9]{Schwede}.  
(The requirement that the second morphism is also an immersion is used to construct an affine cover to which the previous
theorem may be applied locally.  In situations where such a cover can be given by hand, it suffices for one of the morphisms
to be a closed immersion.) 

\begin{example}
Consider a fanifold $\Phi$ with two interior 0-strata, and a single interior 1-stratum joining them.  
Such a fanifold entails the data of a fan $\Sigma$ (for the 1-stratum) and two 
fans $\Sigma', \Sigma''$ for the 0-strata, along with rays $\sigma' \in \Sigma'$ and 
$\sigma'' \in \Sigma''$ and identifications
$$\Sigma'/ \sigma' \cong \Sigma \cong \Sigma''/\sigma''$$
This determines  isomorphisms of toric varieties $\overline{O(\sigma_1)}\cong \bT_\Sigma \cong \overline{O(\sigma_2)}$. 
Thus the maps $\bT({\Sigma_1}) \leftarrow \bT(\Sigma) \rightarrow \bT({\Sigma_2})$ are both closed embeddings, 
so the pushout $\bT(\Phi)$ of the diagram exists as a scheme. 
\end{example}

 \begin{example}
  Let $\Sigma$ be a fan of cones in $\RR^{n+1}$, and consider the
  fanifold $\Phi = \Sigma \cap S^n \subset S^n$ defined in Example \ref{ex:bdry-normalfan}.  
  It is natural to guess that the global sections $\varinjlim_{\Exit(\Phi)^{op}}\bT({F})$ of the cosheaf $\bT$ on this fanifold produces the toric boundary $\partial\bT(\Sigma)$ of the toric variety $\bT(\Sigma).$
  We may check whether or not $\partial\bT(\Sigma)$ agrees with this colimit affine-locally,
  in the restrictions for the standard affine toric charts on $\bT(\Sigma)$.  In each of these,
  the colimit is evidently an iterated pushout of affine varieties along closed embeddings,
  hence by \cite[Theorem 3.4]{Schwede} given by an affine scheme which is the spectrum of the appropriate limit of rings. 
  That the toric boundary has this property is checked in  \cite[Lem. 3.4.1] {GS17}. 
\end{example}

\begin{proposition} \label{prop: existence of lcs} 
If $\Exit(\Phi)$ is equivalent to a poset, then the colimit 
\[\bT(\Phi) := \varinjlim_{\Exit(\Phi)^{op}}\bT({F})\]
exists in the category of schemes.
\end{proposition}
\begin{proof}
We can compute this colimit as a sequence of pushouts.  We begin with the top-dimensional cells of $\Phi$:
the coproduct of $\bT$ over the top-dimensional strata is a disjoint
union of points, which we denote $Z_0$.  
We would like to proceed inductively by defining
\begin{equation} \label{eq: building lcs} 
  Z_k := \varinjlim\left(Z_{k-1} \leftarrow \coprod \partial \bT(F) \rightarrow \coprod \bT(F)\right),
\end{equation}
where the coproduct (disjoint union) is taken over all codimension-$k$ strata $F$. 

The right-hand map in \eqref{eq: building lcs} is an embedding and the left-hand map is finite, which is not quite sufficient to apply 
\cite[Corollary 3.9]{Schwede}.  We can resolve this by instead attaching the $\bT(F)$ one at a time, in which case the
hypothesis that $\Exit(\Phi)$ is a poset implies that the relevant inclusion of $\partial \bT(F)$ 
will be an embedding. 
\end{proof}

Note that the hypothesis of the proposition is inherited by any locally-closed constructible subset of $\Phi$.  
This ensures that each of the iterative pushouts $Z_k$ exists as a scheme.

\begin{example} Consider $\boing$, the stratification of a circle by a point
and an interval, as in Example \ref{ex: necklace}.  The exit path category $\Exit(\boing) = (\bullet \rightrightarrows \bullet)$ is the Kronecker quiver, which is not a poset.  Nevertheless,
the colimit  in question exists. (One could deduce this from \cite[Corollary 3.9]{Schwede} by first deleting a point on 
the $\PP^1$ to reduce to the affine case.)  The result $\bT(\boing)$ is the irreducible nodal curve of genus 1, which is certainly
an example of interest to us. In order to apply the theorems we prove here to this example and others of a similar nature,
one can perform all calculations on a cover and conclude by invoking \'etale descent.  
\end{example}

Let us state in general a result we already used in a special case during the proof of Proposition \ref{prop: existence of lcs}.   

\begin{proposition}
Let $\Phi$ be a fanifold with ideal boundary $\partial_\infty \Phi$.
Assuming $\Exit(\Phi)$ is a poset,
there is a finite morphism of schemes $\bT(\partial_\infty \Phi) \to \bT(\Phi)$.  
\end{proposition} 
\begin{proof} 
The existence of the schemes (and map) follows from Proposition \ref{prop: existence of lcs}.  The 
morphism is finite because it is a finite colimit of embeddings. 
\end{proof} 

\begin{example} \label{ex: halfplane}
  Consider $\Phi = \RR \times \RR_{\ge 0}$, stratified as $(\RR \times \RR_{> 0}) \sqcup (\RR \times \{0\})$.  We equip this with 
  a fanifold structure by putting the fan of $\AA^1$ normal to the ``boundary'' stratum $\RR\times\{0\}$.  Evidently $\bT(\Phi) = \AA^1$.  
The ideal boundary $\partial_\infty \Phi$, which we can picture as an infinite-radius semicircle, is a closed interval, with the fan of $\AA^1$ placed at each endpoint.
Thus $\bT(\partial \Phi) = \AA^1 \sqcup_0 \AA^1$, and the map $\bT(\partial \Phi) \to \bT(\Phi)$ identifies the two $\AA^1$'s. 
\end{example}

\begin{example} \label{ex: unigon} 
  Let $\Phi$ be the 2-disk stratified by a single 0-stratum, 1-stratum, and 2-stratum.  We 
can equip $\Phi$ with fanifold structure by putting the fan of $\AA^2$ at the point, the fan of $\AA^1$ along the interval, 
and the trivial fan on the 2-stratum.  The exit path category of $\Phi$ is $\bullet \rightrightarrows \bullet \rightarrow \bullet$.  
In the iterative construction described in the proof of Proposition \ref{prop: existence of lcs}, $Z_0$ is a point, 
$Z_1$ is $\AA^1,$ and
$$Z_2 = \varinjlim(\AA^1 \leftarrow (\AA^1 \cup_0 \AA^1) \rightarrow \AA^2).$$ 
This colimit corresponds to the fanifold gluing of the fan of $\AA^2$ and the fanifold of Example \ref{ex: halfplane} along
their ideal boundaries, each of which is a closed interval. 

Because everything in sight is affine, we may invoke  \cite[Theorem 3.4]{Schwede} to deduce that the colimit
exists and agrees with the colimit of affine schemes, which  may in turn be computed to be 
$k[x+y, xy, xy^2] \cong k[a,b,c]/(b^3 + c^2 = abc)$.\footnote{We thank David Madore, Laurent Moret-Bailly,
David Speyer, and especially Dan Petersen for help at
\url{https://mathoverflow.net/questions/389117/can-i-glue-the-x-axis-to-the-y-axis}.}   
\end{example}

Let us describe explicitly an affine cover of $\bT(\Phi)$.  Given a stratum $F$ of $\Phi$, we obtain a fanifold 
structure on the closure $\overline{F}$ by restriction.  Assuming that $\Exit(\Phi)$ was a poset, the same holds for 
$\overline{F}$. In this case,  
$\bT(\overline{F})$ is a colimit of affines along embeddings, hence it is itself affine.  
Collecting from all strata $F$ the maps $\bT(\overline{F}) \to \bT(\Phi),$ we obtain an affine (hyper)cover. 

We may make a similar construction even without the hypothesis that $\Exit(\Phi)$ is a poset.  

\begin{proposition} \label{prop: algebraic space}
For $\Phi$ a  fanifold, $\bT(\Phi) := \varinjlim_{\Exit(\Phi)^{op}}\bT({F})$ exists as an algebraic space.
\end{proposition} 
\begin{proof}
For a stratum
$F$, let $\widetilde{F}$ be the ``blowup'' of $\overline{F}$ obtained by replacing the topological boundary of $\overline{F}$ with 
the space of pairs $(f, c),$ where $f \in \partial \overline{F}$ and $c$ is a choice of cone along which $F$ arrives 
in the normal bundle to the stratum containing $f$.  The space $\widetilde{F}$ carries a natural fanifold structure, and the colimit
$\bT(\widetilde{F})$, as an iterated pushout of affines along closed inclusions, is affine.  The natural inclusions among the 
$\bT(\widetilde{F})$ form an \'etale equivalence relation, giving a presentation of $\bT(\Phi)$ 
as an algebraic space.  
\end{proof}

\begin{example}
Consider again the circle $\boing$, stratified by a point $p$ and interval $I$.  Then $\widetilde{I}$ is a closed interval, and 
$\bT(\widetilde{I}) = \AA^1 \sqcup_0 \AA^1$.  The map $\bT(\widetilde{I}) \to \bT(\boing)$ identifies the two copies
of $\AA^1 \setminus 0$ by $z \mapsto 1/z$. 
\end{example}

\begin{remark}
It may be that with additional care, one could find local affine charts by hand. This would allow 
\cite[Theorem 3.4]{Schwede} to be used in place of \cite[Corollary 3.9]{Schwede} in the proof of Proposition 
\ref{prop: existence of lcs}, showing that $\bT(\Phi)$ is a scheme in general. 
\end{remark} 

\begin{remark}
Ideas similar to those of this subsection appear in \cite[Section 2]{GS-log1}.
\end{remark}

\subsection{Coherent sheaves} 
We now turn to studying coherent sheaves on these glued-up objects.  

Let us fix some notation.  For a scheme (or algebraic space, stack, etc.) $Z$, we write $\Coh(Z)$ for the 
category of dg modules over the classical category of coherent sheaves on $Z$.  The theory of this
category is extensively developed in \cite{GR1, GR2}, where it is termed IndCoh.  $\Coh(Z)$ is a presentable dg category.  
We write $\Coh_!: \mathrm{Sch} \to {}^* \mathrm{DG}$ for the functor 
taking $Z \mapsto \Coh(Z)$ and morphisms of varieties to pushforward.  When restricted to proper morphisms, 
this functor lands in ${}^* {}^* \mathrm{DG}$, since pushforward along proper morphisms preserves coherent sheaves
in the ordinary sense. 
We similarly write $\Coh^!: \mathrm{Sch}^{op} \to  \mathrm{DG}^*$ for the corresponding functor 
taking morphisms of varieties to pullbacks; again it lands in ${}^* \mathrm{DG}^*$ when restricted to proper morphisms.  
These functors carry equivalent information; each is obtained from
the other by taking adjoints (of the images of morphisms).   

\vspace{2mm} 

We may then produce a composite functor 
\[
\Coh^! \circ \bT: \Exit(\Phi) \to {}^*\mathrm{DG}^*.
\]
Limits in any of ${}^* \mathrm{DG}$,  $\mathrm{DG}^*$, etc., exist and agree
with the limit in the underlying category of DG categories. Thus, the above composition 
defines a sheaf of categories (independently of our above considerations about when we can take sections of $\bT$). 

Similarly, there is a constructible cosheaf $\Coh_! \circ \bT$, related to the sheaf above by taking adjoints.  This cosheaf is valued in $ {}^* {}^* \mathrm{DG}$, so 
if desired it may be restricted to the
full subcategory of compact objects, which is the (dg) bounded derived category of coherent sheaves in the ordinary sense. 

\begin{proposition} \label{prop: colimits of coh} 
The natural map 
$$(\Coh_! \circ \bT)(\Phi) = \varinjlim_{\Exit(\Phi)^{op}} \Coh(\bT (F)) 
\to \Coh\left(\varinjlim_{\Exit(\Phi)^{op}} \bT (F) \right) = \Coh(\bT(\Phi))$$
is an equivalence.  (Here if necessary we understand $\bT(\Phi)$ as an algebraic space.) 
\end{proposition} 
\begin{proof}
We showed above that the colimit of spaces on the right-hand side can be calculated by iterated pushouts along diagrams where one map is an embedding
and the other is finite.  As the strata are contractible, the left-hand colimit entails the same sequence of pushouts (now taken in ${}^* {}^* \mathrm{DG}$).
By Zariski (or \'etale if $\Exit(\Phi)$ is not a poset) descent, we may reduce the question of checking equivalence to a calculation local
on the cover by the $\bT(\overline{F})$ described above.  

Finally, we apply the result \cite[Chap. 8.A, Thm. A.1.2]{GR2} that pushouts of affine schemes along
diagrams where one inclusion is an embedding and the other is finite are carried to pushouts of coherent sheaf categories.
\end{proof}

We have seen that constructible open subsets of $\Phi$ correspond to closed subsets of $\bT(\Phi)$.  
Taking complements, we can associate an open subset of $\bT(\Phi)$ to a constructible closed subset of 
$\Phi$.  

\begin{proposition}
Writing $\mathrm{Closed}(\Phi)$ for the poset of constructible closed subsets of $\Phi$, where
morphisms are inclusions, we have a functor

\begin{eqnarray*}
\bU: \mathrm{Closed}(\Phi) & \to & \mathrm{Spaces}, \\
\cZ & \mapsto & \bT(\Phi) \setminus \bT(\Phi \setminus \cZ).
\end{eqnarray*} 
If $\Exit(\Phi)$ is a poset, the functor lands in schemes.
\end{proposition}

\begin{example}
For $\Phi = \boing$, the closed constructible subsets are $\bullet$ and $\boing$.   We have 
$\bU(\bullet) = \PP^1 \setminus \{0, \infty\}$, and $\bU(\boing) =  \PP^1 / \{0 = \infty\}$. 
\end{example}

In general, for a closed inclusion of algebraic spaces $S \subset T$, one has an exact sequence
$$\Coh(S) \to \Coh(T) \to \Coh(T \setminus S) \to 0.$$

Taking $S\subset T$ to be the inclusion $\bT(\Phi\setminus\cZ)\subset \bT(\Phi)$ as above, we obtain an exact sequence

\begin{equation} \label{eq: mirror to viterbo} 
(\Coh \circ \bT) (\Phi \setminus \cZ) \to (\Coh \circ \bT)(\Phi) \to (\Coh \circ \bU) (\cZ) \to 0.
\end{equation}

We conclude that the functor $\Coh \circ \bU$ is completely determined from $\Coh \circ \bT$.  It does not
follow formally from the fact that $\Coh^! \circ \bT$ is a sheaf that $\Coh \circ \bU$ satisfies any descent properties, 
but this is nevertheless true:
\begin{proposition}\label{prop:zariski}
Given closed constructible subsets $\cC \cup \cD$ the natural map 
$$ \Coh(\bU (\cC \cup \cD))  \to   \Coh( \bU (\cC)) \underset{\Coh( \bU (\cC \cap \cD))}\times \Coh( \bU (\cD) ) $$
is an isomorphism.  
\end{proposition}
\begin{proof}
This is descent for a Zariski cover. 
\end{proof} 

We wrote $\Coh^*$ rather than $\Coh^!$ in the above proposition to remind ourselves that we are taking pullbacks 
only along open inclusions.   

\begin{remark}
  Proposition \ref{prop:zariski} applies without the hypothesis that $\Exit(\Phi)$ is a poset
but may not be very useful in such cases. For instance, $\boing$ has no nontrivial covers by closed constructible sets.
One can repair this by defining a version of $\bU$ for non-injective maps --- for instance, from the $\widetilde{F}$ used in 
the proof of Proposition \ref{prop: algebraic space}. 
\end{remark}

\section{A-model} \label{sec:lvl}

In this section, we construct the expected SYZ dual space to $\bT(\Phi).$

\begin{theorem} \label{thm: main construction} 
Given a fanifold $\Phi$ (with contractible strata), there is a subanalytic Weinstein manifold $\bW(\Phi)$, a conic subanalytic Lagrangian $\LL(\Phi) \subset \bW(\Phi)$
containing the skeleton of $\bW(\Phi)$, and a map  $\pi: \LL(\Phi) \to \Phi$, with the following properties.
\begin{enumerate}
\item Let $F \subset \Phi$ be a stratum of codimension $d$. Then:
\begin{itemize} 
  \item  $\pi^{-1}(F) \cong F \times T^d$, where $T^d$ is a $d$-torus. 
  \item  $\pi^{-1}(Nbd(F)) \cong F \times \LL_F$, 
where $Nbd(F)$ is an appropriate neighborhood, and $\LL_F$ is the FLTZ Lagrangian associated to the normal fan of $F$.
\item   In a neighborhood of this $F \times T^d$, there is a symplectomorphism of pairs 
$$(\rT^* F \times \rT^* T^d,  F \times \LL_F) \hookrightarrow (\bW(\Phi), \LL(\Phi)).$$  
\end{itemize}
\item 
If $\Phi$ is closed, then $\LL(\Phi)$ is equal to the skeleton of $\bW(\Phi)$.  
\item A sub-fanifold $\Phi' \subset \Phi$ determines a Weinstein subdomain $\bW(\Phi') \subset \bW(\Phi)$ such that 
$\LL(\Phi') \subset \bW(\Phi') \cap \LL(\Phi)$.  
\end{enumerate} 

Finally, $\bW(\Phi)$ carries a Lagrangian polarization given in the local charts by taking the base direction
in $\rT^*F$ and the cotangent fiber direction in $\rT^* T^d$. 
\end{theorem}

\begin{remark} 
In fact the construction (and indeed this entire section) works without the hypothesis of contractible strata, though in that case $\pi^{-1}(F)$
and $\pi^{-1}(Nbd(F))$ will be possibly nontrivial bundles over $F$, with fibers $T^d$ and $\LL_F$, respectively.  
\end{remark}

\begin{remark}
We have chosen our conventions so that $\partial_\infty(\LL(\Phi)) \cong \LL(\partial_\infty \Phi)$.  Indeed, one can see
from our construction that there is a Liouville hypersurface embedding 
$\bW(\partial_\infty \Phi) \subset \partial_\infty \bW(\Phi)$ such that 
$\LL(\Phi)$ is the relative skeleton of the pair $(\bW(\Phi), \bW(\partial_\infty \Phi))$.  
\end{remark}

\begin{remark}
It would be straightforward
to obtain the  germ of $\bW(\Phi)$ along $\LL(\Phi)$ by gluing the putative charts $(\rT^* C \times \rT^* T^d,  F \times \LL_F)$.  
However, to get a exact symplectic form giving a Liouville manifold with core $\Lambda$, one would have to carefully modify 
the local cotangent forms in order to interpolate between them, while preserving the skeleton and maintaining 
the existence of a convex neighborhood thereof. 
 Rather than -- or in order to -- do this, we  construct $\bW(\Phi)$ using the handle attachment process, which already has 
the desired interpolations built in.  
\end{remark}

\begin{remark} \label{rem: lvl SYZ}
  The map $\pi$ is presumably the restriction to the skeleton of a large-volume-limit fibration $\widetilde{\pi}: \bW(\Phi) \to \Phi$.
We expect that with additional care, the construction described in this section would produce this fibration, in fact as an 
integrable system with noncompact fibers.  (This would require checking that the various adjustments made in the handle 
attachment process can be made compatibly with the projection.)    Note that if we fix an appropriate line bundle on 
$\bT(\Phi)$ so that it makes sense to discuss the B-side SYZ base $\Psi$ glued from
moment polytopes, then $\Phi$ is combinatorially dual to $\Psi$. 
\end{remark} 

\begin{remark} \label{rem: takeda section} 
The map $\LL(\Phi) \to \Phi$ has a natural section, glued together from the sections $\LL(\Sigma) \to \Sigma$ given by the inclusion of the cotangent
fiber at zero.  We thank Alex Takeda for this observation. 
\end{remark}

\subsection{Review of Liouville manifolds, skeleta, and gluing}  \label{sec: liouville}

By definition, a {\em Liouville domain} $(W, \omega = d \lambda)$ is an exact symplectic manifold-with-boundary  
such that the Liouville vector field $Z = \omega^\# \lambda$ is outwardly transverse to the boundary, which we denote 
$\partial_\infty W$. This is closely related to the notion of {\em Liouville manifold}, which is a Liouville domain extended by an infinite 
symplectization cone of the form
$\R_{\ge 0} \times \partial_\infty W$.  
The typical examples of Liouville domains and manifolds are codisk and cotangent
bundles, respectively, with the tautological ``$pdq$'' form.   The textbook reference for these notions is \cite{Cieliebak-Eliashberg}.  
For our purposes, the Liouville manifold is the fundamental object, and we use 
domains only to be precise about various intermediate steps of constructions; as is common in the literature, we often pass 
back and forth between Liouville domains and Liouville manifolds without much comment.

\begin{definition}
We say a subset of a Liouville manifold is {\em conic} if it is invariant under the Liouville vector field.  By definition
the {\em skeleton} $\LL_W$ of a Liouville manifold $W$ is the maximal conic compact subset.  
\end{definition}
The Liouville vector field gives 
$W \setminus \LL_W$ a free action of $\RR$, for which the quotient is canonically identified with $\partial_\infty W$.  More generally,
for any conic subset $K \subset W$, we write 
$\partial_\infty K$ for $(K \setminus \LL_W) / \RR$ (or equivalently for its intersection with the boundary of some Liouville subdomain 
completing to the Liouville manifold $W$).  

\begin{definition}
 For a subset $V \subset \partial_\infty W,$ 
 the {\em relative skeleton} associated to $V$ is the subset
 $\LL_{W, V} := \LL_W \sqcup \RR V$. 
 (Usually, we are interested in situations where $V$ is Legendrian and itself the skeleton of a Liouville hypersurface in $\partial_\infty W.$)
\end{definition}

The outward condition
on the vector field creates an evident difficulty with gluing Liouville domains along their boundaries.  However, when 
gluing along a standard neighborhood
of a Legendrian, one can modify the vector field in such a way 
that the resulting glued-up manifold
is again Liouville.  
The modification is canonical up to contractible choices.   
This construction originates in \cite{Weinstein}; we briefly review the idea here. 

Consider a smooth Legendrian $\cL \subset \partial_\infty W$.  
Let us fix such a standard neighborhood 
$$\eta: Nbd_{\partial_\infty W}(\cL)\hookrightarrow J^1 \cL,$$
which we extend to a neighborhood
$$\xi: Nbd_{W}(\cL) \hookrightarrow J^1 \cL \times \RR_{\ge 0}
= \rT^* \cL \times \RR \times \RR_{\le 0} = \rT^* (\cL \times \RR_{\le 0}),$$
chosen so that the Liouville flow on $W$ 
is identified with the translation action on $\RR_{\le 0}$.   

We now define another space $\widetilde{W}$ by first modifying the Liouville flow on $W$ so it is 
carried instead to the cotangent scaling on $\rT^* (\cL \times \RR_{\le 0})$ when sufficiently close to $\partial \cL$, and then
taking conic completion.  
$\widetilde{W}$ is an exact symplectic manifold-with-boundary; in fact,  it is a Liouville sector with exact boundary in the sense of \cite[Section 2]{GPS1}, 
whose skeleton $\LL_{\widetilde{W}}$ is 
the relative skeleton $\LL_{W,\cL}$ defined above.

Note that $\eta$ induces an identification of the {\em actual} (not ideal-at-infinity) boundary of $\widetilde{W}$ with $J^1 \cL$; we denote this also 
$\widetilde{\eta}: \partial \widetilde{W} \xrightarrow{\sim} J^1 \cL$.  Similarly we have 
$$\widetilde{\xi}: Nbd_{\widetilde{W}}(\cL) \hookrightarrow \rT^* (\cL \times \RR_{\le 0})$$ 
which by construction matches the Liouville structure on $\widetilde{W}$ to the standard cotangent scaling, at least over some $(-\epsilon, 0] \subset \RR_{\le 0}$.  
Under the (symplectic but not Liouville-preserving) embedding $W \subset \widetilde{W}$, whose image contains a neighborhood of $\cL$,  we have $\widetilde{\xi}|_{W} = \xi$.

Given Liouville domains $(W, \partial_\infty W)$ and $(W', \partial_\infty W')$, and a smooth manifold $\cL$ with Legendrian embeddings 
$\partial_\infty W \hookleftarrow \cL \hookrightarrow \partial_\infty W'$, we write 
$$W \#_\cL W' := \widetilde{W} \cup_{J^1 \cL} \widetilde{W'}.$$ 
The space $W \#_{\cL} W'$ is a Liouville manifold, but we use the same notation for some domain completing 
to it.  On skeleta we have: 
$$\LL_{W \#_\cL W'} = \LL_{W, \cL} \cup_\cL \LL_{W', \cL}.$$

We will also want to glue together some conical Lagrangians in $W,W'$ to form a new Lagrangian in $W\#_\cL W'.$
\begin{definition} \label{def: biconic} 
We say that a conic subset $\VV \subset W$ is ($\xi$-){\em biconic} if its image under $\xi$ is invariant 
also under the cotangent scaling in the  $\rT^* (\cL \times \RR_{\le 0})$ direction.  Inspection of the deformation used in \cite{Weinstein} 
shows that biconic subsets remain conic in $\widetilde{W}$; we write $\widetilde{\VV} \subset \widetilde{W}$ 
for the saturation under the Liouville flow.  
\end{definition}
It is easy to rephrase the biconicity condition in terms of the standard neighborhood $\eta$ chosen above:

\begin{lemma} \label{lem: biconic}
For $\cL \subset \partial_\infty W$ a smooth Legendrian, and $\VV\subset W$ a (possibly singular) Lagrangian, the 
following are equivalent: 
\begin{itemize}
\item $\VV$ is conic  and $\eta(\partial_\infty \VV) \subset \rT^*\cL \times 0 \subset J^1 \cL$.
\item $\VV$ is biconic.
\end{itemize}
Moreover, in this case $\xi(\VV \cap Nbd_W(\cL)) = \eta(\partial_\infty \VV)|_{\rT^* \cL} \times \RR_{\le 0}  \subset \rT^* \cL \times \rT^* \RR_{\le 0}$.  $\square$
\end{lemma} 

Given biconic Lagrangians $\VV \subset W$ and $\VV' \subset W'$ with matching ends in the sense that 
$\tilde{\eta}(\partial_\infty \VV) = \tilde{\eta}' (\partial_\infty \VV')$, we may form a new Lagrangian
$$\VV\#_{\cL} \VV' := \widetilde{\VV} \cup \widetilde{\VV'}$$
in the glued manifold $W\#_\cL W'.$ 
From the above discussion, we see that the glued Lagrangian $\VV\#_\cL\VV'$ is conic and in a chart near the gluing region is a product 
$\eta(\partial_\infty \VV)|_{\rT^* \cL} \times (-\epsilon, \epsilon)  \subset \rT^* \cL \times \rT^* (-\epsilon, \epsilon)$. 

In this article, we apply the above constructions in the special case 
when $\widetilde{W}' = \rT^*M$ and $\partial M = \cL$.   We term these {\em handle attachments} (though many authors reserve 
this for the case when $M$ is a ball).  For noncompact $M$ we require as usual conicality at infinity.   Note that such a noncompact $M$ 
has both an ordinary boundary $\partial M = \cL$ along which we glue, and an ideal boundary $\partial_\infty M$.  

The gluings we consider will usually occur in the situation where $\cL = (\partial F) \times G$ and we attach $\rT^* F \times \rT^*G$, respecting the product structure. 
In this situation, suppose we are given a conic $\VV \subset W$ which is not only biconic near $(\partial F) \times G$ as 
in Lemma \ref{lem: biconic},
but in addition factors locally as 
$$\eta(\partial_\infty \VV) = \partial F \times \LL \subset \rT^* \partial F  \times  \rT^*G$$ 
for some fixed conic Lagrangian $\LL \subset \rT^*G$.  
\begin{definition}
With $\cL=(\partial F)\times G$ and $\VV\subset W$ as above, the {\em extension of $\VV$ through the handle} 
is the gluing $\VV\#_\cL\VV',$ where we define
$$\VV' = F \times \LL \subset \rT^* F \times \rT^* G.$$ 

For $\Lambda \subset \partial_\infty W$ such that $\eta(\Lambda) \subset \rT^*\cL \times 0 \subset J^1 \cL$, 
the relative skeleton $\LL_{W, \Lambda}$ is biconic.  We say that we 
{\em extend $\Lambda$ through the handle} to mean that we extend $\LL_{W, \Lambda}$ through the handle, and take the boundary 
at infinity of the result.  
\end{definition}
Note that if $\Lambda$ is smooth (with boundary along $(\partial F) \times G$), then so is the extension.  More generally,
other structures or properties of $\Lambda$ which respect the biconic structure can also be extended through the handle.

Let us give a criterion for biconicity:

\begin{lemma} \label{lem: polar} 
Let $E \to M$ be a vector bundle, and $\cL = \partial_\infty \rT_M^* E$ be the conormal Legendrian to the zero section.  
Then there are local coordinates near $\cL$ such that: 

For any collection of submanifolds $S_\alpha \subset E$ which are conic with respect to the scaling of $E$, if 
a Lagrangian $\Lambda\subset T^*E$ is contained in the union  $\bigcup_\alpha \rT^*_{S_\alpha} E$, then $\Lambda$ is biconic along $\cL$.  
\end{lemma} 
\begin{proof}
Note $\rT^*E = E \oplus E^\vee \oplus \rT^* M$ as a bundle over $M$.  We write $P \subset \rT^* E$ for 
the polar hypersurface defined as the kernel of the pairing between $E$ and $E^\vee$.
Then $P$ contains the conormal to any conic subset of $E$, and $\partial_\infty P$ can be locally identified with the cotangent bundle to $\cL$,
compatibly with Liouville structure.  Indeed, the cotangent fibers are locally near $\cL$ the conormals to (codimension-1)
hyperplanes through the origin in fibers of $E$. 

The desired standard coordinates are swept out by applying the Reeb flow to $Nbd_{P}(\cL)$. 
\end{proof}

\begin{corollary} \label{cor: biconic} 
Let $M \subset N$ be a submanifold.  Then the conormal $\rT_M^* N$ admits standard coordinates with respect
to which any conic Lagrangian $\Lambda \subset \rT^*N$ contained in a union of conormals to submanifolds of $M$ is biconic. 
\end{corollary} 
\begin{proof}
Identify a tubular neighborhood of $M$ with a subset of the normal bundle $\rT_M N$. Strata contained in $M$ are (trivially) conic,
so we may apply Lemma \ref{lem: polar}. 
\end{proof} 

Intersections of Legendrians satisfying the hypotheses of the above lemma are quite special.
To illustrate what can go wrong, we
give an example where biconicity cannot be achieved by any choice of coordinates.  

\begin{example}
Consider three lines through the origin in $\RR^3,$ all lying in the same plane $P$, and let $\VV \subset \rT^*\RR^3$ be the union of 
their conormals.  Then $\partial_\infty \VV$ consists of three Legendrian surfaces $C_1, C_2, C_3$ (diffeomorphic to $S^1 \times \RR$) 
whose pairwise intersections $C_i \cap C_j$ consist of two points, namely the two conormal directions to $P$ at the origin.  These two points
are also the intersection of all three surfaces.  

The ideal boundary $\partial_\infty \VV$ is not biconic (for any choice of coordinates) along any of these surfaces.  Indeed, if it were biconic to (say) $C_1$,
then at a triple intersection point $*$, both $C_2$ and $C_3$ would be identified with some conic Lagrangian in $\rT^* C_1$ meeting $C_1$ only
at $*$.  Since the only such conic Lagrangian is the cotangent fiber, this would imply that locally $C_2 = C_3$, which is contradicted by the definition of the $C_i$.  
By the same reasoning, no configuration of Legendrians in which three smooth components pass through the same point and are pairwise transverse
can be biconic, in the sense above, with respect to any of these components.  

Note that the above reasoning does not disallow such a configuration 
of Lagrangians from being contained in the skeleton of a Liouville hypersurface.  (For instance, begin with a ball carrying the radial Liouville form, take
three mutually transverse linear Lagrangians through the origin, attach handles along their ideal boundaries, and then contactize.) 
In fact, the constructions above can be made more generally in neighborhoods of such Liouville hypersurfaces (which take the place 
of the cotangent bundle of a smooth Legendrian); for instance, standard models for the gluing of Liouville manifolds along such Liouville hypersurfaces can be found in \cite{Avdek}, \cite[Section 2]{GPS1}, \cite[Section 3.1]{Eliashberg},
\cite[Section 9]{GPS3},  \cite[Section 2]{AGEN3}.
\end{example}

\subsection{Review of FLTZ Lagrangian} \label{sec: FLTZ}

Fix the usual data necessary to define a toric variety: a rank $n$ lattice $M$ and a rational polyhedral fan 
$\Sigma \subset M_\RR:=M\otimes_\ZZ \RR$.  
Consider the $n$-torus 
$$\widehat{M} := \Hom(M, S^1) = M_\RR^\vee/M^\vee.$$  
Note the canonical isomorphism $\rT^* \widehat{M} = \widehat{M} \times M_\RR$.  
For any subset  $Z \subset M$, we write $Z^\perp \subset \widehat{M}$ for the locus of maps carrying $Z$ to 
$1 \in S^1$.  

  For each cone $\sigma$ in $\Sigma,$ we write $\LL_\sigma$ for the Lagrangian
  \begin{equation} \label{eq: Lcone}
	\LL_\sigma = \sigma^\perp\times \sigma\subset \widehat{M} \times M_\RR=\rT^* \widehat{M}.
  \end{equation}
  For a cone $\sigma$ we write $\partial_\infty \sigma$ for its projectivization, so that 
  $\partial_\infty \LL_{\sigma} = \sigma^\perp \times \partial_\infty \sigma$. 

 The union   of these conic Lagrangians is the 
  {\em FLTZ Lagrangian} 
  \[
  \LL(\Sigma) :=\bigcup_{\sigma\in\Sigma}\LL_\sigma.
  \]

 Recall that  
for a cone $\sigma$ in the fan $\Sigma$, we write $\Sigma/\sigma$ for the normal fan to $\sigma$ in   
$M/\sigma$.  Using the canonical identification $\widehat{M / \sigma} \cong \sigma^\perp$, we consider $\LL({\Sigma/\sigma}) \subset \rT^*\sigma^\perp.$
The Legendrian $\partial_\infty \LL(\Sigma)$ admits biconic local coordinates respecting the geometry of these quotient fans: 

\begin{lemma}\label{lem:fltz-inductive}
  For a fan $\Sigma$, there is a system of standard coordinates 
   $$\eta_\sigma: N_\sigma =  Nbd(\partial_\infty\LL_\sigma) \hookrightarrow J^1 \partial_\infty\LL_\sigma = \rT^* \partial_\infty\LL_\sigma \times \RR,$$
  indexed by cones $\sigma \in \Sigma,$
     for which $\LL_\Sigma$ is biconic:
  $\eta_\sigma(\partial_\infty\LL(\Sigma) \cap N_\sigma ) \subset \rT^* \partial_\infty\LL_\sigma \times 0$. 

 Moreover, for any cone $\tau$ containing $\sigma$ in its closure, the Legendrian boundary of $\LL_\tau$ has an expression in $N_\sigma$ as a lower-dimensional FLTZ Lagrangian:
  \begin{equation}\label{eq:fltz-inductive}
 \eta(\partial_\infty \LL_\tau \cap N_\sigma)  =
\LL_{\tau/\sigma}\times \partial_\infty \sigma \times 0 \subset 
\rT^*\sigma^\perp\times \rT^* \partial_\infty \sigma \times 0 = 
\rT^*\partial_\infty \LL_\sigma \times 0.
  \end{equation}
 The coordinates $\eta_\sigma$ determine for each cone $\sigma$ a Liouville hypersurface 
 \[R_\sigma := \eta_\sigma^{-1}(\rT^* \partial_\infty\LL_\sigma \times 0)
 \subset \partial_\infty\rT^*\widehat{M}
 \]
 containing $\partial_\infty \LL_\sigma$ as its skeleton, and we may choose these coordinates to ensure that 
 for $\sigma\subset \bar{\tau}$ as above, we have 
 $R_\tau \cap N_\sigma \subset R_\sigma$. 
\end{lemma}
\begin{proof}  
   First observe that if the intersection $\overline{\LL_\tau} \cap \LL_\sigma = (\tau^\perp \times \overline{\tau}) \cap (\sigma^\perp \times \sigma)$ is ever nonempty, 
   then the intersection $\overline{\tau} \cap \sigma$ must itself be nonempty, so that  
   $\sigma$ is contained in the closure of $\tau.$ Thus, for each $\sigma,$ it is necessary to prove the biconicity near $\partial_\infty \LL_\sigma$ of $\LL_\tau$ for such $\tau.$ Note that in this case, there is a contravariant inclusion $\sigma^\perp\supset \tau^\perp.$
   Moreover, in this case we will deduce the inductive characterization of the Legendrian $\partial_\infty\LL_\tau$ in terms of the FLTZ Lagrangian $\LL_{\tau/\sigma}$ from the fact that near $\sigma,$ the cone $\tau$ can be described as a product $\tau/\sigma\times \sigma,$ so that the Lagrangian
   $\LL_\tau = \tau^\perp\times \tau$ locally looks like $( (\tau/\sigma)^\perp \times (\tau/\sigma) )\times \sigma.$ (Note that for $\bar{\tau}\supset \sigma,$ the tori $(\tau/\sigma)^\perp$ and $\tau^\perp,$ which live inside respective tori $\widehat{M/\sigma}\subset \widehat{M}$, are actually the same torus.)

	 Since the Lagrangian $\LL_\sigma$ is contained in the conormal bundle $\rT^*_{\sigma^\perp}\widehat{M},$ we may produce standard coordinates for $\LL_\sigma$ using the method of Corollary \ref{cor: biconic}: restrict to a tubular neighborhood
	 \begin{equation*}
	 U_\sigma:=Nbd(\sigma^\perp)\subset \widehat{M}
 \end{equation*}
	 of $\sigma^\perp$, which may be identified with a subset of the normal bundle $N_{\sigma^\perp}\widehat{M},$ and then apply the polar hypersurface construction of Lemma \ref{lem: polar} to produce coordinates on the boundary of
	 \[
		 \rT^*_{\sigma^\perp}\widehat{M} = \rT^*_{\sigma^\perp}(N_{\sigma^\perp}\widehat{M})
		 =
		 (N_{\sigma^\perp}\widehat{M})\times (N_{\sigma^\perp}\widehat{M})^\vee \times T^*\sigma^\perp.\]
		 By Corollary \ref{cor: biconic}, the Lagrangian $\LL_\tau,$ which is contained in the conormal to $\tau^\perp\subset \sigma^\perp,$ will be biconic in these coordinates.

		 Unfortunately, the standard coordinates so constructed are not compatible as we range over cones in the fan $\Sigma$: 
		 For $\bar{\tau}\supset\sigma,$ we constructed a polar hypersurface $P_\sigma$ in the cotangent bundle of a neighborhood of $\sigma^\perp$ as the zero set of the function
		 \begin{equation}\label{eq:polar-pairing}
			f_\sigma: \rT^*U_\sigma = (N_{\sigma^\perp}\widehat{M})\times (N_{\sigma^\perp}\widehat{M})^\vee \times \rT^*\sigma^\perp\to \RR
		 \end{equation}
		 which pairs the first two factors. Near $\tau^\perp\subset \sigma^\perp,$ we constructed a hypersurface $P_\tau = \{f_\tau=0\}$ in the analogous way, but there is no inclusion between $P_\tau$ and (restriction near $\tau^\perp$ of) $P_\sigma,$ since the function $f_\tau$ contains more terms than $f_\sigma,$ corresponding to normal directions to $\tau^\perp$ which are contained in $\sigma^\perp.$

		 We will therefore need to modify our polar hypersurface construction. For each cone $\sigma$ in $\Sigma,$ we continue to write $U_\sigma\subset \widehat{M}$ for the tubular neighborhood of $\sigma^\perp$ in $\widehat{M}$, and we denote by $V_\sigma\subset \rT^*\widehat{M}$ a conic tubular neighborhood of $\LL_\sigma$ projecting to $U_\sigma$ under the projection $\rT^*\widehat{M}\to \widehat{M},$ chosen moreover so that $V_{\sigma},V_{\sigma'}$ are disjoint when $\bar{\sigma}\cap\bar{\sigma}'=\{0\}.$ We will also denote by $W_\sigma\subset M_\RR$ the image of $V_\sigma$ under the cotangent fiber projection $\rT^*\widehat{M_\RR}\to M_{\RR}.$

		 For each $V_\sigma,$ we will define a function 
		 $
			 g_\sigma:V_\sigma\to \RR,
		 $
		 a modification of the polar hypersurface function $f_\sigma$ described above, such that the zero loci of the restrictions $g_\sigma|_{V_\sigma\cap V_\tau}$ and $g_\tau|_{V_\sigma\cap V_\tau}$ agree and, by replacing the polar hypersurface $P_\sigma = \{f_\sigma=0\}$ with $P_\sigma':=\{g_\sigma=0\}$ in the polar hypersurface construction, we still obtain coordinates in which $\LL_\sigma$ is biconic. We will therefore obtain ribbons $R_\sigma$ which are compatible with each other, in the sense that $R_\tau\cap N_\sigma\subset R_\sigma,$ as desired.

We define $g_\sigma$ inductively on the dimension of the cone $\sigma.$ For $\sigma$ a 1-dimensional cone, we take
\[g_\sigma:=f_\sigma|_{V_\sigma},\]
the restriction to $V_\sigma$ of the polar hypersurface pairing $f_\sigma:\rT^*U_\sigma\to \RR$ defined in \eqref{eq:polar-pairing}.

Now let $\tau$ be a 2-dimensional cone spanned by rays $\sigma_1$ and $\sigma_2.$ We need to extend the function $(g_{\sigma_1},g_{\sigma_2}):V_{\sigma_1}\sqcup V_{\sigma_2}\to \RR$ to a function which is also defined on $V_\tau.$
To accomplish this, observe that any nonzero linear combination $(a_1 f_{\sigma_1}+a_2 f_{\sigma_2})|_{V_\tau}:V_\tau\to \RR,$ with $a_1,a_2\in\RR_{\geq 0},$ still defines a hypersurface $P_\tau'$ with the desired properties,
and this remains true if we allow $a_1(\bar{m}),a_2(\bar{m})$ to vary in the cosphere coordinate $\bar{m}\in \partial_\infty W_\tau$:
in other words, the hypersurface 
\[
P_\tau':=\{a_1(\bar{m})f_{\sigma_1}+a_2(\bar{m})f_{\sigma_2} = 0 \}
\]
contains the conormal to any submanifold in $\tau^\perp$ (since the functions $f_{\sigma_i}$ both vanish there), guaranteeing that components $\LL_{\rho}$ of the FLTZ skeleton with $\overline{\rho}\supset \tau$ are biconic with respect to the coordinates defined by $P_\tau'.$

Moreover, by taking $a_2(m)\equiv 0$ in $V_{\sigma_1}$ and $a_1(m)\equiv 0$ in $V_{\sigma_2}$ (and both $a_i$ nonzero in the intermediate region between the $V_{\sigma_i}$), we obtain a function 
\[
	g_\tau:=a_1(\bar{m})f_{\sigma_1}+a_2(\bar{m})f_{\sigma_2}: V_\tau\to \RR
\]
such that the zero locus of $g_i|_{V_{\sigma_i}}$ agrees with $\{g_{\sigma_i}=0\}$. Similarly, for each higher-dimensional cone $\rho$, we continue to interpolate among nonzero linear combinations of the functions $g_\sigma$ defined for lower-dimensional cones to produce the function $g_\rho.$
\end{proof}

We recall below in Theorem \ref{thm: toric mirror symmetry} the role of $\LL(\Sigma)$ in homological mirror symmetry \cite{FLTZ2, Ku, GS17},
and in Example \ref{ex: very affine} its corresponding appearance as the relative skeleton of the Liouville sector associated to a  Hori-Vafa superpotential
\cite{GS17, Zhou-skel}.   $\LL(\Sigma)$ also arises directly from considerations around SYZ mirror symmetry; see \cite{FLTZ1}.

\subsection{Proof of Theorem \ref{thm: main construction}} \label{sec:lvl-defn} 
Now we have the ingredients we need to prove Theorem \ref{thm: main construction}. The proof will proceed exactly as described in Remark \ref{rem: handle attachment fanifold}: we begin with the fanifold $\Phi_0$ of (neighborhoods of) vertices in $\Phi,$ which will contribute to $\bW(\Phi)$ a disjoint union of cotangent bundles of tori, equipped with FLTZ Lagrangians; the edges in $\Phi$ will specify pieces of these FLTZ Lagrangians corresponding to 1-dimensional cones, which we will glue together via handle attachment, and then we will extend the remaining pieces of the FLTZ Lagrangians across the handle attachment by biconicity. The same procedure is followed for at the next step for the 2-dimensional cones, and so on.

This construction is best summarized in Figure \ref{fig:2d-leg-gluing} below.

\begin{proof}[Proof of Theorem \ref{thm: main construction}] \label{sec: proof} 
The fanifold $\Phi$ has a filtration $\Phi_0 \subset \Phi_1 \subset \cdots \subset \Phi_n = \Phi$ by fanifolds $\Phi_k$ defined as neighborhoods
of the $k$-skeleta $\mathrm{Sk}_k(\Phi)$; we will prove the theorem inductively over the fanifolds $\Phi_k$.

At the beginning of stage $k$, we will already have $\bW(\Phi_{k-1})$ and $\pi: \LL(\Phi_{k-1}) \to \Phi_{k-1}$ satisfying the conditions
of the theorem.  We will then have to construct $\bW(\Phi_{k})$ and $\pi: \LL(\Phi_{k}) \to \Phi_{k}$.  As we have noted in Remark \ref{rem: handle attachment fanifold},
$\Phi_{k}$ is constructed from $\Phi_{k-1}$ by handle attachment.  We will lift this fanifold handle attachment to a Weinstein handle
attachment on $\bW(\Phi_{k-1})$ to form $\bW(\Phi_k),$ and we will then extend $\LL_{k-1}$ through the newly attached handle to form $\LL_{k}$. 

At stage ${k}$, for each interior $k$-stratum $F \subset \Phi$, there is a smooth closed Legendrian 
\begin{equation} 
\cL_F := \pi^{-1}(F) \cap \partial_\infty \LL(\Phi_{k-1}).
\end{equation}
We write $F_\circ := F \setminus \pi(\Phi_{k-1})$.  This $F_\circ$ is a manifold-with-boundary, where the boundary 
is the portion of the ideal boundary of $F$ which is in the interior of $\Phi$.  (That is, $\partial F_\circ$ is 
what we called $\partial_{in} F$ in Remark \ref{rem: handle attachment fanifold}.)  

The local description of $\pi$ ensures that $\cL_F \cong \partial F_\circ \times \widehat{M}_F$, where
$M_F$ is the rank $(n-k)$ lattice associated to the
stratum $F$, and $\widehat{M}_F$ is the corresponding Pontrjagin dual $(n-k)$-torus.  
 Thus we may attach a handle  $\rT^* F_\circ \times \rT^* \widehat{M}_F$. 

We will show below that 
show that $\cL_F$ admits local coordinates in which $\LL(\Phi_{k-1})$ is locally biconic and in fact 
splits locally as a product $ \partial F_\circ \times  \LL_F \subset \rT^* \partial F_\circ \times \rT^* \widehat{M}_F$.   
Having done so, we may extend $\LL(\Phi_{k-1})$ through this handle. 

We do this for all k-strata, so in total our handle attaching locus is 
$$ \cL_{k} := \partial  (\pi_k^{-1}(\mathrm{Sk}_{k} \Phi)) =  \coprod_{\substack{\mathrm{interior}\, F \\ \dim(F) = k}} \partial_\infty \LL \cap  \pi^{-1}(F), $$
where the union is taken over the $k$-strata $F$ of $\Phi$.  

We define $\pi$ on the handle $\rT^* F_\circ \times \rT^* \widehat{M}_F$ as the product of projections to $F_\circ$ and to 
the cotangent fibers of $\rT^*\widehat{M}_F$.  
The compatibility condition on fan structures ensures that the restriction of this projection from $\LL(\Phi_k)$ to $\LL(\Phi_{k-1})$ agrees with the projection already defined there.

Finally, we must return to the point we postponed above: the demonstration that at each stage, the $\cL_F$ are as advertised and 
that $\LL(\Phi_{k-1})$ have the appropriate properties along them.  
For expository reasons we give steps 1 and 2 
explicitly, although they are special cases of the general procedure at step $k$.  (By step 2 one sees essentially the full
complexity of the construction.)  It may be helpful to read 
these steps while referring to Example \ref{ex:skel-basic-2d} and Figure \ref{fig:2d-leg-gluing} below.

\vspace{2mm}

{\bf Step 0}:  
In the case $\Phi = \Sigma$ is simply a rational polyhedral fan for the lattice $M \subset M \otimes \R$, 
we define $\bW(\Sigma) = \rT^* \widehat{M}$ and take $\LL(\Sigma)$ to be the FLTZ Lagrangian. 
The map $\pi$ is just the projection to cotangent fibers. 

In a general fanifold, $\Phi_0$ is isomorphic to a disjoint union of (disk neighborhoods of the origin in) such fans.  We define 
$\bW(\Phi_0), \LL(\Phi_0)$ by the corresponding disjoint unions, and similarly $\pi: \LL(\Phi_0) \to \Phi_0$. 

\vspace{2mm} 
{\bf Step 1}: 
$\cL_1$ is a disjoint union of  Legendrians $\cL_F$, indexed by interior 1-strata $F \subset \Phi_1$.
Given such a stratum $F$, let $F', F''$ be the 0-strata in its closure, and let 
$\sigma' \subset \Sigma'$ and $\sigma'' \subset \Sigma''$ be the rays associated to $F$ in the respective fans of 
$F'$ and $F''$.  The corresponding components of $\cL_F$ are the 
Legendrians $\partial_\infty \LL_{\sigma'}$ and $\partial_\infty \LL_{\sigma''}$ 
in the respective cosphere bundles $\partial_\infty \rT^*\widehat{M}_{F'}$ and $\partial_\infty \rT^* \widehat{M}_{F''}$.    
As the boundary $\partial F_\circ$ consists of 0, 1, or 2 points, there is an evident diffeomorphism $\cL_F \cong \partial F_\circ \times \widehat{M}_F$.  
Note that since $F$ is a 1-stratum, $\widehat{M}_F$ is an $(n-1)$-torus. 

Recall that we define $\bW(\Phi_1)$ by the gluing
$$\bW(\Phi_1) := \bW(\Phi_0) \#_{\cL_1} \coprod_{\substack{\mathrm{interior}\, F \\ \dim(F) = 1}} \rT^*F \times \rT^* \widehat{M}_F. $$ 

We now check the properties of $\cL_1$ which allow us to extend $\LL(\Phi_0)$ through the handle.  
In Lemma \ref{lem:fltz-inductive}, we identified standard coordinates near $\cL_F$ for which $\LL(\Sigma')$ and $\LL(\Sigma'')$ are biconic. 
We should verify that $\LL(\Phi_0)$  locally factors respecting the product structure 
$\rT^* \widehat{M}_F \times \partial F_\circ$.  There is only something to check in case $\partial F_\circ$ is two points, over which lie 
$\LL(\Sigma'/\sigma')$ and $\LL(\Sigma''/\sigma''),$ respectively.  Thus, the
local factorization follows from the fan compatibilities 
$$\Sigma'/\sigma' = \Sigma_F = \Sigma''/\sigma''$$
required in the definition of a fanifold. 

We conclude that $\LL(\Phi_0)$ extends through the handle, and we define $\LL(\Phi_1)$ 
to be this extension.   

\vspace{2mm} 

{\bf Step 2}. 
$\cL_2$ is a disjoint union of  Legendrians $\cL_F$, indexed by interior 2-strata $F \subset \Phi_2$.
To each exit path $H \to F$ of 0-2 strata, we obtain a two-dimensional cone $\sigma_{H \to F} \subset \Sigma_{H}$
and corresponding Legendrian  $\partial_\infty \LL_{H \to F}$ in $\partial_\infty \rT^* \widehat{M}_H$. 

The collection $\bigsqcup_{H, F} \partial_\infty \LL_{H \to F}$ is a smooth Legendrian-with-boundary  in $\partial_\infty \bW(\Phi_0)$, and
we must show that it extends through the handles we attached in forming $\bW(\Phi_1)$.  
As it is a subset of $\LL(\Phi_0)$, we need only check that it respects the local factorization
already established for $\LL(\Phi_0)$.  

To see this, consider now a flag $G \to F$ of 1-2 strata.  This flag specifies a ray $\sigma_{G \to F} \subset \Sigma_G$.  
If $H \to G \to F$ and $H' \to G \to F$ are two flags of 0-1-2 strata extending $G \to F$,  then the fanifold compatibility 
conditions require the fan isomorphisms
$$\Sigma_H / \sigma_{H \to G} = \Sigma_G = \Sigma_{H'} / \sigma_{H' \to G}$$
to identify the images of the corresponding cones
$$\sigma_{H \to F} \mapsto \sigma_{G \to F} \mapsfrom \sigma_{H' \to F}.$$
This shows that indeed $\cL_F^{pre} := \bigsqcup_{H, F} \partial_\infty \LL_{H \to F}$ factors as advertised, and it is straightforward
to see that its extension through the handles is indeed $\cL_F$.  

Lemma \ref{lem:fltz-inductive} gives the biconicity and local product structure of $\LL(\Phi_0)$ along 
$\cL_F^{pre}$. Note now that in Lemma \ref{lem:fltz-inductive}, these structures were constructed inductively on the dimension of cone $\sigma,$ so that near the region of the previous handle attachment, the biconic coordinates for $\partial_\infty\LL_{H\to F}$ and $\partial_\infty \LL_{H'\to F}$ will agree. Therefore this entire structure extends through the handles to give the corresponding structures for 
$\LL(\Phi_1)$ along $\cL_F$.

\vspace{2mm}

{\bf Step $\mathbf{1} \le \mathbf{k} \le \mathbf{n}$}: 
We have already constructed $\bW(\Phi_{k-1})$, $\LL(\Phi_{k-1})$, and $\pi: \LL(\Phi_{k-1}) \to \Phi_{k-1}$.  
Our task is to study $\cL_{k}$, which is by definition the disjoint union over interior $k$-strata of 
$\cL_F := \pi^{-1}(F) \cap \partial (\LL(\Phi_{k-1}))$.  It suffices to study each $\cL_F$ independently, so we fix 
some interior $k$-stratum $F$. 

As in Step 2, the Legendrian $\cL_F$ can also be described by beginning with all exit paths
$H \to F$ from $0$-strata to $F$,  considering the Legendrians $\partial_\infty \LL_{H \to F}$ corresponding to $k$-cones
$\sigma_{H \to F} \in \Sigma_H$, and iteratively extending these through all handle attachments associated to flags
of strata ending in $F$.  Existence of these extensions follows as in Step 2 from compatibility of fans, and 
the result is readily seen to be $\cL_F$.    

Local biconicity of $\LL(\Phi_k)$ along $\cL_F$ and the factorization (locally near $\LL_F$) of $\LL(\Phi_k)$ as
$$\LL(\Phi_k) = F_\circ \times \LL({\Sigma_F})   \subset  \rT^* F_\circ \times \rT^* \widehat{M}_F$$
follow by extending through the handles the corresponding facts 
(proven in Lemma \ref{lem:fltz-inductive})
for the original FLTZ Legendrians
$\partial_\infty \LL_{H \to F}.$  
Once again, the inductive construction of Lemma \ref{lem:fltz-inductive} ensures that these structures agree near previously attached handles, so that we may extend them over the handles.
\end{proof}

\subsection{Examples} 
We now give some examples of the Weinstein manifold $\bW(\Phi)$ constructed by Theorem \ref{thm: main construction} for some interesting fanifolds $\Phi.$

\begin{example}
For any fan $\Sigma$, the space $\bW(\Sigma)$ is the cotangent bundle of a torus, and $\LL(\Sigma)$ is the FLTZ Lagrangian
for $\Sigma$.  An inclusion $\Sigma' \subset \Sigma$ corresponds to an inclusion $\LL(\Sigma') \subset \LL(\Sigma)$ of 
FLTZ Lagrangians. 
\end{example}

\begin{example} \label{ex:skel-basic-1d}
  As in Example \ref{ex: necklace}, consider the fanifold $\Phi$ associated to the stratification of $S^1$  into $r$ intervals and $r$ points. 
  Let us step through the construction of the corresponding $\bW(\Phi)$.  At step zero, we associate to each point the cotangent bundle 
  of a circle, $\rT^* \widehat{\ZZ}$;  the $\bW(\Phi_0)$ will be the disjoint union of these.  (We use the $\widehat{\ZZ}$ in part for consistency
  with the above, and in part to distinguish this circle from the circle $\Phi = S^1$.) 
  Inside the $\rT^*  \widehat{\ZZ}$ we have the FLTZ skeleton mirror to $\PP^1$; 
  the union of these is the $\LL(\Phi_0)$.   The 
  Legendrian at infinity is the positive and negative conormals over $0 \in T^1$.  The union of all of these gives the $\cL_1$.  
  At step one, we attach 1-handles, attaching the positive conormal point of one $T^* T^1$ to the negative conormal point of the next.  
  The procedure terminates here.  Note $\bW(\Phi)$ is the Weinstein manifold obtained from a compact 2-torus by deleting $r$ points. 
  This space is well known to be mirror to the necklace of $\PP^1$s which is $\bT(\Phi)$. 
\end{example} 

\begin{figure}[h]
  \begin{center}
    \includegraphics[width=3in]{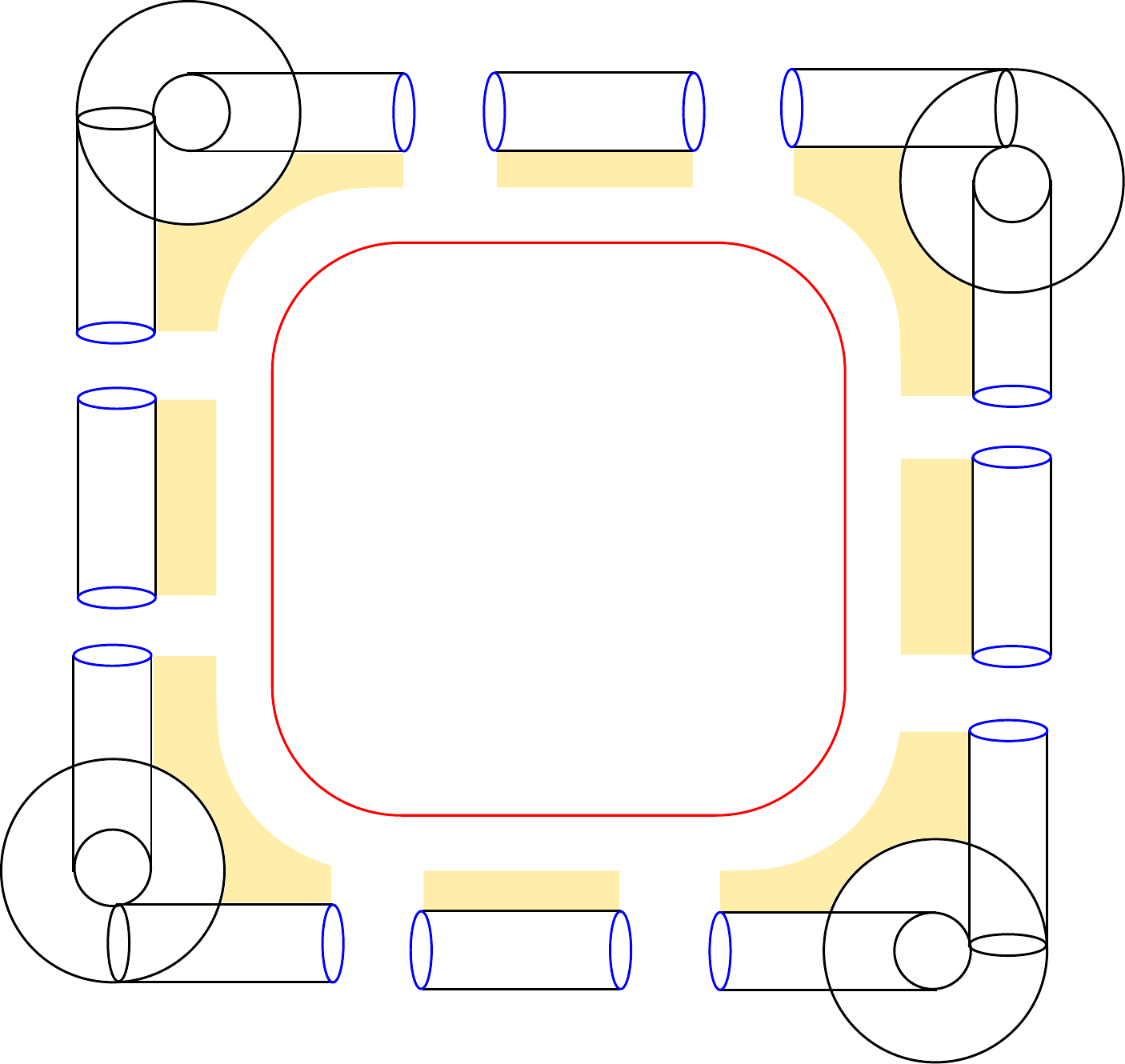}
  \end{center}
  \caption{For $\Phi = [0,1] \times [0,1]$, the space $\LL(\Phi_0)$ consists of the four depicted corner components.  Each 
  consists of the union, in the cotangent bundle to a 2-torus, of the zero section, the positive conormals to the longitude and meridian, and 
  a quadrant of the conormal to the intersection point of the longitude and meridian (this last drawn in yellow).  
  The blue locus in their boundaries is $\cL_1$.  Extending $\LL(\Phi_0)$ through the corresponding handles gives rise 
  to the edge components; attaching all these together gives $\LL(\Phi_1)$.  The red locus indicates the Legendrian 
  $\cL_2$, along which we will attach a 2-disk in the final step.}
  \label{fig:2d-leg-gluing}
\end{figure}

\begin{example}\label{ex:skel-basic-2d}
  Consider $[0,1]\times [0,1]$ with stratification by interior, boundary edges, and boundary vertices.  The normal geometry to each vertex is naturally
  identified with a fan for $\AA^2$ placed at each vertex. (The fan of $\AA^2$ spans a quadrant of $\RR^2,$ and to each vertex we associate a fan spanning the appropriate quadrant.) 
  The normal geometry to each edge is naturally identified with the fan of $\AA^1$. 
  The Lagrangian skeleton of the resulting Weinstein manifold 
  is obtained from the gluing depicted in Figure~\ref{fig:2d-leg-gluing} after attaching a 2-disk along the red Legendrian.  
\end{example}

\begin{example} \label{ex: very affine} 
Consider a fan $\Sigma \subset \RR^{n+1}$.  Assume the fan is simplicial, and that 
the primitive generators of each ray lie on the boundary of some fixed convex polytope $\Delta^\vee$.  Consider 
the fanifold $\Phi := \Sigma \cap S^n$ as in Example \ref{ex:bdry-normalfan}.  
The calculations of \cite{GS17, Zhou-skel} can be interpreted as showing that in this case  $\bW(\Phi)$ 
is (a tailoring of) a generic hypersurface  $H\subset (\CC^*)^{n+1}$ with Newton polytope $\Delta^\vee$. 
\end{example} 

\begin{example}\label{ex:3A1}
Let $\Sigma\subset \RR^3$ be the standard fan of cones for the toric variety $\AA^3,$ and let $\mathring{\Sigma}$ be the fan obtained from $\Sigma$ by deleting the rays. The fan $\mathring{\Sigma}$ has a fanifold structure inherited by the fanifold $\Phi:=\mathring{\Sigma}\cap S^2,$ which is a ``2-simplex without its vertices.''
The Lagrangian $\LL(\Phi)$ is a union of three cylinders and a 2-simplex (with vertices removed) with each edge glued to one of the three cylinders, as illustrated in Figure~\ref{fig:3P1s}.

\begin{figure}[h]
  \begin{center}
    \includegraphics[width=3in]{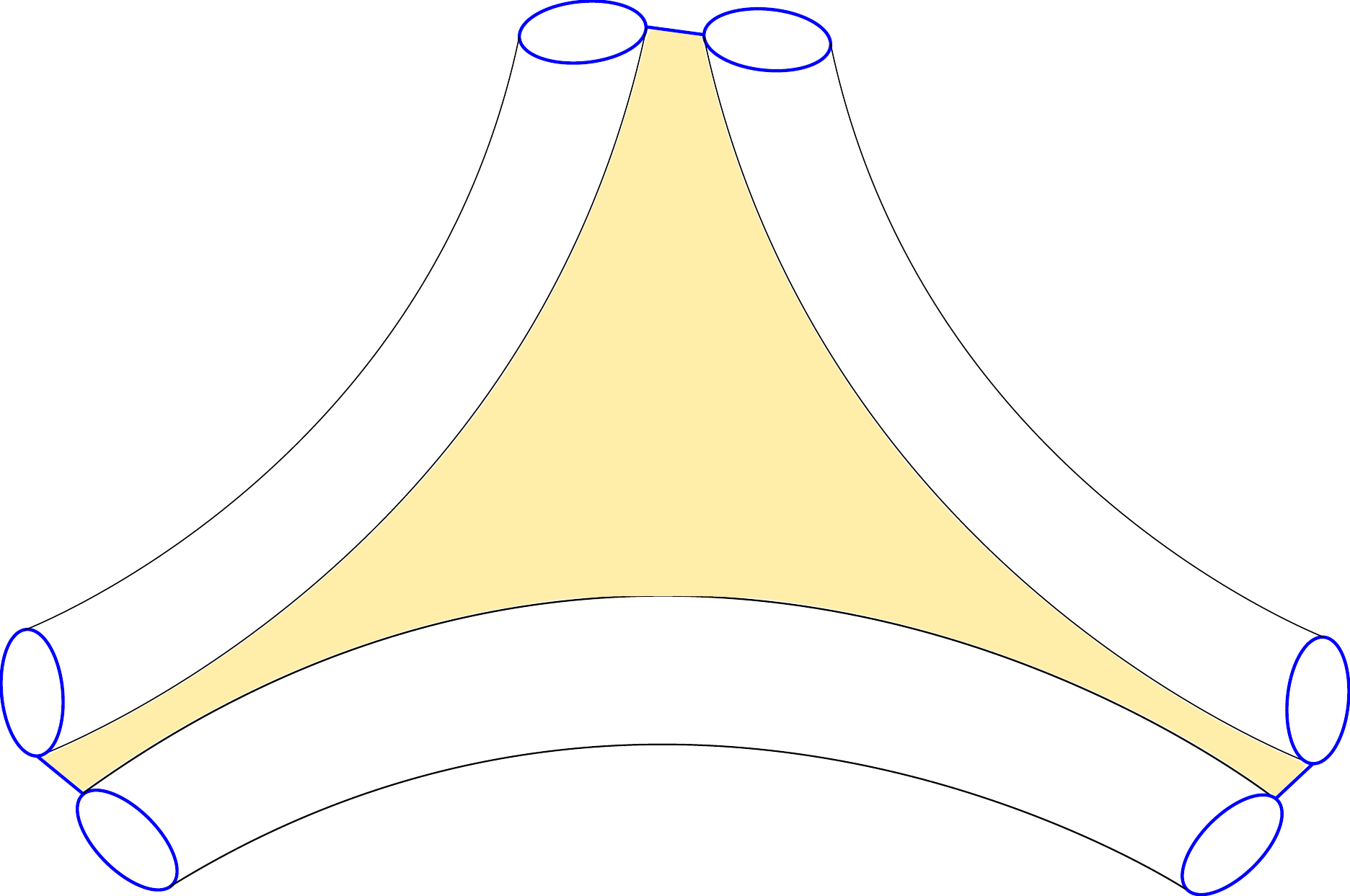}
  \end{center}
  \caption{The Lagrangian skeleton $\LL(\Phi)$ 
  mirror to the coordinate axes in $\AA^3$.  
  The boundary skeleton $\LL(\partial_\infty \Phi) = \partial_\infty \LL(\Phi)$ is depicted in blue; it is mirror to three disjoint copies
  of the coordinate axes in $\AA^2$.}
  \label{fig:3P1s}
\end{figure}

We can see that the fanifold $\Phi$ has three 1-strata, each one equipped with the fan of $\AA^1,$ and one 2-stratum shared among them, so that the large-complex-structure variety $\bT(\Phi)$ determined by $\Phi$ is three copies of $\AA^1$ meeting at a point --- the union of the coordinate axes in $\AA^3.$ This example shows that by deleting strata from a fanifold $\Phi,$ we can produce $n$-dimensional B-side varieties whose singularities are more complicated than those occurring in $n$-dimensional toric geometry.
\end{example}

\subsection{Microsheaves} \label{sec: microsheaves}
Having constructed $\bW(\Phi)$ and $\LL(\Phi)$, and noting that the polarization gives rise to the necessary
Maslov data to define the Fukaya category (as described for instance at \cite[Sec. 5.3]{GPS3}), we write $\Fuk(\bW(\Phi), \partial \LL(\Phi))$ for
the category of modules over what would in \cite{GPS2} be called the (partially) wrapped Fukaya category.  (That is, we take 
the presentable DG category associated to the usual Fukaya category.)

Recent work on localization of Fukaya categories \cite{GPS1,  shende-microlocal, GPS2, GPS3, NS20} ends in an equivalence 
 \cite[Thm. 1.4]{GPS3} between Fukaya categories of this sort and the global sections of a certain constructible sheaf of categories
 obtained from microlocal sheaf theory: 
$$\Gamma(\LL(\Phi), \mu sh_{\LL(\Phi)})^{op} \cong \Fuk(\bW(\Phi), \partial \LL(\Phi)).$$
Here, the $\mu sh_{\LL(\Phi)}$ is a constructible sheaf of categories on $\LL(\Phi)$
valued in ${}^* \mathrm{DG}^*$.  Taking opposite category is an artifact from various conventions and can be absorbed into, e.g., 
negating the symplectic form.  In the cotangent bundle setting, one can absorb it into negating $\LL(\Sigma)$, and indeed it is the Lagrangian
$-\LL(\Sigma)$ which appears in \cite{FLTZ1,FLTZ2, Ku} --- although $\LL(\Sigma)$ itself appears in \cite{Nadler-wrapped, GS17} where the actual relative skeleton for a toric mirror is computed.

We may compute global sections after first taking the pushforward $\pi_* \mu sh_{\LL(\Phi)}$, which is a constructible
sheaf of categories over $\Phi$ itself; it is in terms of this pushforward that we later formulate our mirror symmetry results. 

The remainder of the present subsection is dedicated to explaining how we may compute the pushforward sheaf of categories $\pi_*\mu sh_{\LL(\Phi)}$ in practice.
Our previous work \cite{GS17} can be understood
as a computation of the restriction of $\pi_* \mu sh_{\LL(\Phi)}$ to the neighborhood of a stratum.  By itself, this seems insufficient to determine
the sheaf $\pi_* \mu sh_{\LL(\Phi)}$, since such a determination would also require knowledge of the gluing isomorphisms on overlaps.  However,
we are in possession of the pleasant fact that 
the charts in Theorem \ref{thm: main construction}
are all cotangent bundles with their canonical polarizations. 
Below, we will explain how this fact allows us to reconstruct the desired gluing isomorphisms.

We first recall the basic definitions of the microsheaf theory \cite{Kashiwara-Schapira} and its 
globalization \cite{shende-microlocal, NS20}. 

Let $M$ be a differentiable manifold and 
consider the category $\Sh(M)$ of sheaves on $M$ valued in 
some fixed symmetric monoidal presentable DG category, which we may as well take to be $\Mod(k)$.   To  $F \in \Sh(M)$, 
there is a conical locus $ss(F) \subset \rT^*M$ of codirections along which the local space of sections of $F$ is not constant. 
The textbook reference is \cite{Kashiwara-Schapira}; see any recent article (for instance \cite{NS20}) for 
comments on updates to the homological algebra foundations.  

Because $ss(F)$ interacts well with sums, products, and cones, it is natural
to consider for any given conic $\Lambda \subset \rT^*M$ the category $\Sh_\Lambda(M)$ of sheaves whose microsupport is
contained in $\Lambda$.  We are typically interested in $\Lambda$ which are subanalytic and the closure
of their smooth Lagrangian points; we term such a $\Lambda$ {\em singular Lagrangian}, or just Lagrangian.  
Any smooth Lagrangian point of $\Lambda$ determines a ``microstalk functor'' $\Sh_\Lambda(M) \to \Mod(k)$, or more precisely a
family of such functors parameterized by some topological data at the point and differing by tensor product
with invertible objects of $\Mod(k)$.  (If there is a Lagrangian disk in $M$ tranverse to this point of $\Lambda$, then it 
is carried to a corepresentative of this functor by \cite{GPS3}.) 

It is a deep result that for any sheaf $F$, its singular support $ss(F)$ is coisotropic \cite[Theorem 6.5.4]{Kashiwara-Schapira}. 
One useful application of this fact is the following:
if one knows in advance that $ss(F)$ is contained in some conic (singular) Lagrangian $\Lambda$ 
but is in fact disjoint from the smooth locus of $\Lambda$, then $ss(F)$ is empty.  It follows that the microstalk functors
at smooth points of $\Lambda$ generate  $\Sh_\Lambda(M)$ when $\Lambda$ is singular Lagrangian. 

In fact, $\Sh_\Lambda(M)$ is the global sections of a sheaf of categories $\mu Sh$
over $\Lambda$: 
\begin{definition}\label{defn:mush1}
For $\Lambda$ a conic Lagrangian, the sheaf of categories $\mu Sh_\Lambda$ is defined as the 
sheafification of the following presheaf of categories on $\rT^*M$:
\begin{equation} \mu Sh^{pre}_\Lambda(U) := \Sh_{\Lambda \cup (\rT^*M \setminus U)} (M) /  \Sh_{\rT^*M \setminus U} (M). \end{equation} 
(Here we intentionally write $\mu Sh$ rather than $\mu sh$ to distinguish between this and a different construction which we will recall below.)
Evidently $\mu Sh_\Lambda$ is conic and is the pushforward of a sheaf supported on $\Lambda$, which we denote also $\mu Sh_\Lambda$.  
By restriction away from
the zero section, we obtain a sheaf of categories on the Legendrian $\partial_\infty\Lambda\subset \partial_\infty \rT^*M$.
\end{definition}

Now let $\Lambda$ be an arbitrary Legendrian, carrying
the germ of a contact manifold $U$ in which it is embedded as a Legendrian.   The basic innovation of \cite{shende-microlocal}
was to consider positive codimension 
embeddings $U \hookrightarrow \partial_\infty \rT^*M$.  Such an embedding realizes $\Lambda$ as a subcritical isotropic, so that
coisotropicity of microsupports implies that the microsheaf category
$\mu Sh_\Lambda$ is actually 0.
However, we can remedy this by thickening $\Lambda.$
\begin{definition}\label{defn:mush2}
  Let $\widetilde{\Lambda} \subset \partial_\infty \rT^*M$ be a Legendrian obtained by thickening $\Lambda$ along a choice of stable polarization of the symplectic normal bundle of $U$. Then we define a sheaf of categories $\mu sh_{\Lambda}$ on $\Lambda$ by restriction of the sheaf $\mu Sh_{\tilde{\Lambda}}$ defined above:
\[ \mu sh_{\Lambda} := \mu Sh_{\tilde{\Lambda}}|_{\Lambda}.
\]
\end{definition}

There is a canonical stabilization functor $(\Lambda \subset U) \mapsto (\Lambda \times \RR \subset U \times \rT^* \RR)$, 
and a canonical isomorphism $\mu sh_{\Lambda} \cong \mu sh_{\Lambda \times \RR}|_{\Lambda \times 0}$, 
induced from the
canonical isomorphism $Sh_{\rT_\RR^* \RR}(\RR) \cong \mathrm{Mod}(k)$. 
By Gromov's h-principle, the space of all such embeddings in $\RR^{2n+1}$  as $n \to \infty$ is (nonempty and) arbitrarily connected.  
By contact invariance of microsheaves, one sees therefore that $\mu sh_{\Lambda}$ depends only on the stable normal
polarization (in the sense that the space of further choices is contractible). 

\begin{remark} While homotopic choices of stable normal polarization give equivalent 
(sheaves of) categories $\mu sh_\Lambda$, the space of such choices is {\em not} contractible; one can formulate this 
universally in terms of the existence of a (canonical up to contractible choice) 
sheaf of categories over the Lagrangian Grassmannian  bundle of the stable normal bundle, 
locally constant in the bundle direction.  In fact, this sheaf descends from the Lagrangian
Grassmannian to a principal $\mathrm{BPic}(\Mod(k))$ bundle \cite[Sec. 10]{NS20}.  Thus, the true requirements for
defining $\mu sh_{\Lambda}$ are a trivialization of this bundle: when $k = \ZZ$, this can be seen to be the same 
topological data as is usually required to define gradings and orientations for the Fukaya category.
(See for instance \cite[Sec. 5.3]{GPS3} for details.) 
We will not need this descent here, as we will have a natural choice of polarization available to us. 
\end{remark} 

Let us also recall that a Lagrangian polarization of the tangent bundle of a symplectic manifold (or of the contact distribution
of a contact manifold) defines a stable normal polarization by asking that in some cosphere bundle embedding, the given
tangent polarization is contained in the ambient cosphere polarization, with the quotient defining the normal polarization; such an
embedding  exists by h-principle considerations. 

That is, to compute $\mu sh_{\LL(\Phi)}$ directly 
from the definition involves finding an embedding of $\bW(\Phi)$ as a (possibly high-codimensional)
Liouville hypersurface of $\RR^{2n+1}$ and then studying the front projection of $\LL(\Phi)$ to $\RR^{n+1}$.  
While the h-principle guarantees that it is possible to find such an embedding, it is not clear how
one would do so in practice.  Instead, we will 
take advantage of the fact (Theorem \ref{thm: main construction})  that $\bW(\Phi)$ is covered by cotangent bundle charts, 
locally with respect to which $\LL(\Phi)$ is  conical. 

At this point, a conical Lagrangian $\Lambda\subset \rT^*M,$ carries two sheaves of categories:
\begin{itemize}
  \item Definition \ref{defn:mush1} defines the sheaf of categories $\mu Sh_\Lambda$ on the conical Lagrangian $\Lambda.$
  \item The Lagrangian $\Lambda$ admits a natural embedding as a Legendrian in the contactization $\rT^*M\times \RR.$ The fiber polarization on $\rT^*M$ determines a polarization on the contazctization, so that Definition \ref{defn:mush2} determines a canonical (up to contractible choices) sheaf of categories $\mu sh_\Lambda$.
\end{itemize}

Our definitions so far do {\em not} determine an identification of these sheaves of categories.  We now fix such a choice.
Consider the diagram: 
$$M \xleftarrow{\pi} M \times (0, \infty) \xrightarrow{j} M \times \RR$$
and the map $j_* \pi^*: Sh(M) \to Sh(M \times \RR)$.   Note that $\partial_\infty ss(j_* \pi^* \ZZ)$ is the positive (in the $\RR$ direction) conormal 
$\widetilde{M} := \partial_\infty T^+ (M \times 0)$ to $M.$ Projection to the base gives an identification $\widetilde{M} \xrightarrow{\sim} M \times 0$.   
In fact, there are standard coordinates $\eta: Nbd(\widetilde{M}) \to J^1 M = \rT^*M \times \R$ 
such that for any sheaf $F$ on $M$, there is a local factorization
$$\eta(\partial_\infty ss(j_* \pi^* F)) =   ss(F) \times 0 \subset \rT^*M \times \R.$$ 
Given a conic Lagrangian $\Lambda \subset T^*M,$ which we may also consider as a Legendrian $\Lambda \times 0  \subset  \rT^* M \times \RR$,  
the map $j_* \pi^*$ induces an equivalence
$\eta^* \mu Sh_\Lambda \cong  \mu sh_{\eta^{-1}( \Lambda)}$ 
 of sheaves of categories on $\Lambda$.
Note that the restriction $\eta^* \mu Sh_{\partial_\infty \Lambda} \cong  \mu sh_{\eta^{-1}(\partial_\infty \Lambda)}$ 
of this isomorphism to the boundary of $\Lambda$
agrees with the (previously chosen) stabilization isomorphism used in the definition of $\mu sh$, because near the boundary $\partial_\infty \Lambda,$ the relation between $\Lambda$ and $\Lambda\times 0$ is precisely 
the standard stabilization.

More generally, for a vector space $V$ and open strictly convex cone $\iota: C \hookrightarrow V$,  we may  
consider the analogous diagram
\begin{equation}\label{eq:diagram-fixediso}
M \xleftarrow{\pi} M \times C  \xrightarrow{j} M \times V.
\end{equation}

Writing $C^\vee$ for the dual cone inside the cotangent fiber $T_0^* V$ , there are standard coordinates 
$$\partial_\infty \rT^*(M \times V) \supset Nbd(\partial_\infty (M \times C^\vee)) \xrightarrow{\eta}   \rT^* M \times  \rT^* \partial_\infty C^\vee \times \RR$$
such that for any sheaf $F$ on $M$, we have a local factorization of singular supports
$$\eta(\partial_\infty ss(j_* \pi^* F)) =   ss(F) \times \partial_\infty C^\vee \times 0 \subset \rT^* M \times  \rT^* \partial_\infty C^\vee \times \RR.$$ 
Consider the conic Lagrangian $\Lambda \times \partial_\infty C^\vee \subset \rT^*M\times \rT^* \partial_\infty C^\vee$. In the coordinates $\eta,$ we have an equivalence
of sheaves of categories
\begin{equation}\label{eq:choice-iso1}
  \eta^* \mu Sh_{\Lambda \times \partial_\infty C^\vee} \cong  \mu sh_{\eta^{-1}( \Lambda \times \partial_\infty C^\vee)}.
\end{equation}
The $C^\vee$ factor is contractible and is the zero section of its cotangent bundle,
so that we have a canonical isomorphism.  
\[p^* \mu Sh_{\Lambda}  \cong  \mu Sh_{\Lambda \times \partial_\infty C^\vee},\] 
where $p: \Lambda \times \partial_\infty C^\vee \to \Lambda$ is the projection,
and likewise we have
a canonical isomorphism 
\[\mu Sh_{\Lambda}  \cong  p_* \mu Sh_{\Lambda \times \partial_\infty C^\vee}.\] 
Combining these with the equivalence \eqref{eq:choice-iso1} we chose above, we have therefore produced equivalences
\begin{equation} \label{eq: cotangent charts} \eta^*  p^* \mu Sh_{\Lambda} \cong  \mu sh_{\eta^{-1}( \Lambda \times \partial_\infty C^\vee)}, \qquad \qquad 
\mu Sh_{\Lambda} \cong p_* \eta_* \mu sh_{\eta^{-1}( \Lambda \times \partial_\infty C^\vee)}. \end{equation}

The significance of Equation (\ref{eq: cotangent charts}) is that the LHS is computed in some specific cotangent bundle, 
while the RHS {\em depends (up to contractible choice) only upon the germ of stable contact embedding and normal polarization}.  
Whenever in any contact manifold we find  $\Lambda \times C^\vee$ with some chart 
$\eta: Nbd(\Lambda \times C^\vee) \to \rT^*M \times \rT^* C^\vee \times \RR$, and the fixed normal polarization restricts to the standard normal polarization in this chart
(as is the case for instance if we define the normal polarization by a tangent polarization restricting in this chart to the standard normal polarization ---
such a polarization entails the base and fiber polarizations $\rT^* C^\vee$ and $\rT^*M,$ respectively), then we obtain fixed isomorphisms as in 
Equation (\ref{eq: cotangent charts}).  This will be the key tool in our computation of global microsheaf categories.

\begin{remark} \label{rem: choice} 
The equivalence \eqref{eq:choice-iso1}, and therefore also the equivalences of \eqref{eq: cotangent charts}, are not canonical, in the sense that 
they depend on the choice we made to produce them through the diagram \eqref{eq:diagram-fixediso}. Nevertheless, we fix this choice once 
and for all, so that from here on, we do have a fixed way of identifying these sheaves of categories.

The reason for our particular choice is the following. 
Recall from \cite[Chap. 3.7]{Kashiwara-Schapira} 
the Fourier-Sato transform 
\[\mathfrak{F}: Sh(M \times V) \to Sh(M \times V^\vee),\] 
defined as the integral transform with kernel given by 
the polar locus $\{(v, v^\vee) \ge 0\}$.   Fix a point $c^\vee \in C^\vee$.  Then for any $F \in Sh(M)$, there is a canonical
isomorphism $F \cong \mathfrak{F}(j_* \pi^* F)|_{M \times c^\vee}$.  That is, $|_{M \times c^\vee} \circ \mathfrak{F}$ is a left 
inverse to $j_* \pi^*$. 

Our choice is designed to match the corresponding choice in \cite[Definition 4.3.1]{Kashiwara-Schapira}, so that Lemma \ref{lem: sato}
below takes its stated form. 
\end{remark}

Suppose now that we have a submanifold $M \subset N$, let us assume with trivial normal bundle
$\rT^*_M N = \rT^*_m N \times M$.  Suppose we are given some conic Lagrangian $\LL_N \subset \rT^* N$, 
which, when restricted to an appropriate choice of tubular neighborhood for $M$, is also conic
for the scaling action on the tubular neighborhood.  Lemma \ref{lem: polar} therefore ensures that $\LL_N$ is
biconic along the Legendrian $\rT^*_M N$ in the sense of Definition \ref{def: biconic}.  Assume in
addition that for some open cone $C^\vee \times M \subset \rT^*_m N \times M$, there is a chart
$\eta: Nbd(\partial_\infty C^\vee \times M) \hookrightarrow \rT^* \partial_\infty C^\vee \times \rT^* M \times \RR$
such that $\eta(\partial_\infty \LL_N) = C^\vee \times \LL_M \times 0$. 

Then the previous discussion determines an isomorphism
$$p_* \eta_* \mu sh_{\LL_N}|_{Nbd(\partial_\infty C^\vee \times M)} \xrightarrow{\sim} \mu sh_{\LL_M}$$
and, passing to global sections, a particular morphism
$$Sh_{\LL_N}(N) = \Gamma(\LL_N, \mu sh_{\LL_N}) \to 
\Gamma(Nbd(\partial_\infty C^\vee \times M), \mu sh_{\LL_N}) 
= \Gamma(\LL_M, \mu sh_{\LL_M} ) = Sh_{\LL_M}(M).$$
It is an exercise to show: 
\begin{lemma} \label{lem: sato} 
This morphism $Sh_{\LL_N}(N) \to Sh_{\LL_M}(M)$ is naturally isomorphic to composition of the Sato microlocalization along $M$ (as defined in \cite[Chap. 4.3]{Kashiwara-Schapira})
with the restriction to $c^\vee \times M,$ for any $c^\vee \in C^\vee$. 
\end{lemma}

We studied a particular instance of this in \cite[Lemma 7.2.2]{GS17}.  
There we showed that sending a fan to the category of sheaves microsupported in the 
corresponding FLTZ skeleton in fact extends to a functor which we now term
\begin{eqnarray*}
fsh: \mathrm{Fan}^{\twoheadrightarrow} & \to & {}^* \mathrm{DG}^*, \\
\Sigma & \mapsto & Sh_{\LL(\Sigma)}(T_\Sigma).
\end{eqnarray*} 
The maps on morphisms are constructed using the standard charts on $\partial_\infty \LL_\Sigma$ (described here in  Lemma \ref{lem:fltz-inductive}), 
from which one sees that if $\sigma \subset \Sigma$ 
is a cone, then the Sato microlocalization (the composition of specialization to the normal cone with the Fourier-Sato transform) 
along $\sigma^\perp$, followed by projecting out the trivial $\sigma$ factor, gives a map
$$
\mu_{\sigma^\perp}: 
Sh_{\LL(\Sigma)}(T_\Sigma) \to 
Sh_{\LL(\Sigma/\sigma)}(T_{\Sigma/\sigma}).
$$ 

As a fanifold $\Phi$ includes the data of a map $\Exit(\Phi) \to \mathrm{Fan}^{\twoheadrightarrow}$, we may compose 
with $fsh$ to get a map, which we also call 
$fsh: \Exit(\Phi) \to {}^* \mathrm{DG}^*$.  

\begin{proposition} \label{prop: slippery} 
For any fanifold $\Phi$, there is an equivalence $\pi_* \mu sh_{\LL(\Phi)} \cong fsh$  
of sheaves of categories over $\Phi$. 
\end{proposition}
\begin{proof}
Both sheaves of categories $\pi_* \mu sh_{\LL(\Phi)}$ and $fsh$ can be described in terms of the images
in $\Phi$ of the cover discussed in Theorem \ref{thm: main construction}, and the corresponding overlaps.  
But all charts in this cover are of the form $\rT^*M \times \rT^* F$, where $\bL(\Phi)$ is some conic in $\rT^*M$ 
times the zero section
in $\rT^*F$, and the polarization is the fiber direction in $\rT^*M$ times the base direction in $\rT^*F$.  We have 
seen this gives fixed identifications of the corresponding sections of 
$\pi_* \mu sh_{\LL(\Phi)}$ and $fsh$.  
Moreover, 
all restriction maps from the standard charts are either trivial (i.e. induced from the restriction of the contractible $F$ to a contractible open subset),
or of precisely the kind we have just seen correspond to the defining Sato microlocalizations of $fsh$.   
\end{proof}

\subsection{Viterbo restriction}

In general, given a Weinstein subdomain $W' \subset W$, there is a Viterbo restriction functor 
$\Fuk(W) \to \Fuk(W')$.  (We use the definition given in \cite[Sec. 8.3]{GPS2}, which is conjecturally 
equivalent to the partially defined functor of \cite{AS} on the domain of definition of the latter.)  This 
functor is the quotient by the cocores of $W$ which are not contained in $W'$ \cite[Prop. 8.15]{GPS2}.  

Consider the category $\mathrm{WeinSubDom}$ whose morphisms are inclusions of Weinstein subdomains.  
We write 
\[
\Fuk^*: \mathrm{WeinSubDom}^{op} \to {}^* {}^* \mathrm{DG}
\]
for the contravariant functor taking inclusions of subdomains to Viterbo restriction of Fukaya categories
(which preserves compact objects because it is defined before 
taking module categories).  

Given a fanifold $\Phi$, recall that we write $\mathrm{Closed}(\Phi)$ for the poset of constructible closed sets and inclusions among them.  
From Theorem \ref{thm: main construction} (3), we have a functor 
\[
\bW:\mathrm{Closed}(\Phi) \to \mathrm{WeinSubDom},
\] and by composition with $\Fuk^*$, we can
obtain a  functor 
\[
\Fuk^* \circ \bW: \mathrm{Closed}(\Phi)^{op} \to {}^* {}^* \mathrm{DG}.  
\]

Now consider any closed $\Phi' \subset \Phi$.  
From \cite[Prop. 8.15]{GPS2} and the comparison \cite{GPS3}, we have a commutative diagram where the rows are exact: 

\begin{equation} \label{eq: viterbo} 
\begin{tikzcd}
\mu sh_{\LL(\Phi)} (\LL(\Phi) \setminus \LL(\Phi'))^{op} \arrow[r, "\eta_!"] \arrow[d, equal] & \mu sh_{\LL(\Phi)} (\LL(\Phi))^{op}  \arrow[d, equal] \arrow[r]  &  \mu sh_{\LL(\Phi')} (\LL(\Phi'))^{op} \arrow[r] \arrow[d, equal] & 0  \\ 
\langle \text{Cocores of } \bW(\Phi) \setminus \bW(\Phi') \rangle  \arrow[r] & \Fuk(\bW(\Phi)) \arrow[r, "v"] & \Fuk(\bW(\Phi')) \arrow[r] & 0
\end{tikzcd}
\end{equation} 

In the above diagram, the lower-right map $v$ is Viterbo restriction, and the upper-left map $\eta_!$ is the left adjoint to the natural restriction of microsheaves.

\section{Homological mirror symmetry at large volume} \label{sec:hms}

By now, given a fanifold $\Phi$, we have produced two constructible sheaves of categories: 
$\Coh^! \circ \bT$, defined from the algebraic geometry of toric varieties, and $\pi_* \mu sh_{\LL(\Phi)},$ defined from 
symplectic geometry and microlocal sheaf theory.   Now we compare them.  

The basic ingredient is mirror symmetry for toric varieties.  
In the framework of microlocal sheaf theory, mirror symmetry for toric varieties 
was formulated in \cite{FLTZ2} and proven in \cite{Ku}.  Crucial to our approach
is a functoriality result established in \cite{GS17}, matching restriction to orbit closures
with microlocalization.  Let us formulate these results in our current terminology.  

In this section, we assume all fans are smooth.  However, we will remove this hypothesis
in Theorem
\ref{thm: singular}, so we will leave 
it out of the theorem statements. 

\begin{remark}
The (temporary) restriction to smooth fans has to do with fact that in both \cite{FLTZ2} 
and \cite{GS17}, calculations are made using a certain collection of objects which, 
on the B-side, are quasi-coherent sheaves.  
To proceed using these objects in general would required finding appropriate ind-coherent lifts.  
Although this is presumably possible, it is not the strategy of proof in \cite{Ku}. Rather,
in \cite{Ku} the result in the smooth case is used to deduce 
the corresponding result in the general case by descending along toric blowups.  
In Theorem \ref{thm: singular}, we will imitate this strategy to remove
the hypothesis of smoothness from the result of \cite{GS17} 
\end{remark} 

\begin{theorem} \label{thm: toric mirror symmetry} \cite{Ku, GS17}
The functors $fsh^{op}$ and $\Coh^! \circ \bT$ from $\mathrm{Fan}^{\twoheadrightarrow} \to {}^* \mathrm{DG}^*$, 
are equivalent.  
\end{theorem}
\begin{proof}
  For smooth fans, \cite{FLTZ2} gives a morphism $\Coh(\bT(\Sigma)) \hookrightarrow Sh_{\LL(\Sigma)}(T_\Sigma)^{op} =: fsh(\Sigma)^{op}$,
  and 
this morphism is proven by \cite{Ku} to be an isomorphism.  For smooth fans, compatibility with the structure of functors 
out of $\mathrm{Fan}^{\twoheadrightarrow}$ follows from 
the comparison of \cite[Lemma 7.2.1]{GS17} and \cite[Lemma 7.2.2]{GS17}.  
\end{proof}

\begin{theorem} \label{thm: local mirror symmetry}
There is an equivalence of sheaves of categories on $\Phi$:
$$\Coh^! \circ \bT \cong \pi_* \mu sh_{\LL(\Phi)}^{op}.$$ 
\end{theorem}
\begin{proof}
This follows by composing Theorem \ref{thm: toric mirror symmetry} with  
the map $\Exit(\Phi) \to \mathrm{Fan}^{\twoheadrightarrow}$, using Proposition \ref{prop: slippery} to identify $fsh$ with $\pi_*\mu sh_{\LL(\Phi)}.$ 
\end{proof} 

\begin{theorem} \label{thm: global mirror symmetry} 
There is an equivalence of categories $\Coh(\bT(\Phi)) \cong \Fuk(\bW(\Phi), \partial \LL(\Phi))$.
\end{theorem}
\begin{proof}
We conclude this by taking global sections of the comparison in Theorem \ref{thm: local mirror symmetry}, 
using Proposition \ref{prop: colimits of coh} to compute the left-hand side and the comparison between microsheaves
and Fukaya categories \cite{GPS3} for the right-hand side.   
\end{proof}

\begin{remark}
The proof of \cite[Theorem 7.4.1]{GS17} amounts to the special case when $\Phi = S^n \cap \Sigma$. 
In that setting, as everything in sight was embedded into a cotangent bundle, we did not need the 
constructions of \cite{shende-microlocal, NS20}, and correspondingly did
not need Proposition \ref{prop: slippery}. 
\end{remark} 

\begin{remark}
The isomorphism of Theorem \ref{thm: global mirror symmetry} takes the section mentioned in Remark \ref{rem: takeda section}
to the structure sheaf, as follows by gluing together analogous (known) statement in the case of toric varieties.  This 
is as one would expect from the SYZ picture. 
\end{remark}

\begin{remark}
The categorical Calabi-Yau structure plays a key role in the proposal to extract higher-genus enumerative invariants from the Fukaya category
\cite{Costello}, and thus to pursue this direction it would be desirable to show that mirror symmetry is compatible with Calabi-Yau structures.
In this situation, the local-to-global formalism of \cite{shende-takeda} provides a natural framework for doing so.  Indeed, 
when all fans are smooth and complete, the various (all isomorphic) constructible sheaves of categories on fanifolds we have produced here are locally saturated, 
so that the main result of \cite{shende-takeda} provides a local Calabi-Yau structure on $\pi_* \mu sh_{\LL(\Phi)}$.  
\end{remark}

We turn to compatibility with Viterbo restriction.  Recall that we write $\mathrm{Closed}(\Phi)$ for the poset of closed constructible subsets. 

\begin{corollary} \label{cor: viterbo functoriality} 
There is an equivalence 
\[\Coh^* \circ \bU \cong \Fuk^* \circ \bW\]
of contravariant functors from
$\mathrm{Closed}(\Phi)$ to ${}^* {}^* \mathrm{DG}$. 
\end{corollary} 
\begin{proof}
  Compare the short exact sequence in \eqref{eq: mirror to viterbo} to the short exact sequence in \eqref{eq: viterbo}.
Theorem \ref{thm: local mirror symmetry} gives a functorial matching of the first two terms; 
hence we obtain one for the third. 
\end{proof} 

\begin{corollary}\label{cor: seidel locality} 
For closed $\Phi$ covered by closed subsets $\Phi_\alpha$, the map from $\Fuk(\bW(\Phi))$ 
to the limit
$$\varprojlim \bigg( \prod_\alpha \Fuk( \bW(\Phi_\alpha)) \to \prod_{\{\alpha, \beta\}}  
\Fuk(\bW(\Phi_\alpha \cap \Phi_\beta)) \to \prod_{\{\alpha, \beta, \gamma \}}  
\Fuk (\bW(\Phi_\alpha \cap \Phi_\beta \cap \Phi_\gamma)) \to \cdots \bigg)$$
is an isomorphism.  Here the $\bW(\Phi_\alpha)$ (etc.) are Weinstein subdomains and the maps are Viterbo restrictions. 
\end{corollary} 
\begin{proof}
This is Zariski descent translated across Corollary \ref{cor: viterbo functoriality}. 
\end{proof}

\begin{remark} 
Corollary \ref{cor: seidel locality} was in some form  suggested by Seidel \cite{Seidel-speculations} and 
verified by Heather Lee \cite{Lee} by geometric methods in the case of Riemann surfaces.  We emphasize
that this local-to-global principle is {\em not} the same as that of \cite{GPS2}.
\end{remark}

Finally, let us mention certain twists of our constructions, which introduce geometric deformations
on one side, and gerbes on the other.  

\begin{remark}
There is a deformation which is geometric on the B-side and gerby on the A-side.  When gluing
toric varieties on the B-side, we can twist the gluing by an automorphism induced from the torus action.  This is given by data
on the double overlaps, subject to compatibility conditions on the triple overlaps.  The corresponding construction on 
the A-side is to twist $\mu sh$ as follows: on the double overlaps, the topology of the skeleton retracts to a torus, and we may
twist $\mu sh$ by tensor product with a local system on this torus (corresponding to multiplying by an element of the mirror dual 
algebraic torus).  Again this is data on double overlaps, and compatibility conditions on triple overlaps.  So we see that a B-side geometric
deformation corresponds to an A-side gerbe. 
\end{remark}

\begin{remark}
Another twist is gerby on the B-side and geometric on the A-side.  When gluing coherent sheaf categories of B-side varieties, 
we could twist the result by specifying a 
line bundle on each codimension-1 stratum and using it to twist the gluing. These choices of line bundles must satisfy a compatibility condition along
codimension-2 strata.  On the A-side, note in \cite{shende-toric} one finds that at least for a smooth fan $\Sigma$, the 
FLTZ Lagrangian $\LL(\Sigma)$ comes in a noncharacteristic family over a real torus $\Pi_\Sigma$.  
Now over a 1-stratum $I$ in $\LL(\Phi)$, we could replace $\LL(\Sigma_I) \times I$ by
a 1-parameter family of skeleta parametrized some loop $I \to \Pi_\Sigma$.  Being able to continue and attach 2-strata imposes 
a compatibility condition.  In fact, the fundamental group of the torus $\Pi_\Sigma$ can be naturally identified with 
$\mathrm{Pic}(\bT(\Sigma))$, and the monodromies in the family are mirror to the autoequivalence 
of $\Coh(\bT(\Sigma))$ given by tensor product with the corresponding line bundle.  Thus, the microsheaf category of this twisted skeleton is mirror 
to the twisted coherent sheaf category described above.  We note that in this twisted construction there is no longer 
a section of $\LL(\Phi) \to \Phi$. 
\end{remark}

\section{Singular and stacky fans} \label{sec: singular}

We now remove the smoothness hypothesis. 

\begin{theorem} \label{thm: singular}
The results of Section \ref{sec:hms} hold without any smoothness hypothesis on the fans. 
\end{theorem}
\begin{proof}
It suffices to free Theorem \ref{thm: toric mirror symmetry} from the smoothness hypothesis.  We will do so working directly
with $\mu sh$ in 
place of $fsh$, as we are free to do by Proposition \ref{prop: slippery}.  Basically the point is that
we can embed the question into one involving only smooth fans by taking toric blowups.   

Let $\Sigma$ be a fan.  
Recall that a toric blowdown $\pi: \bT(\Sigma') \to \bT(\Sigma)$ 
corresponds to a subdivision of cones: the cones of $\Sigma$ are subdivided to form those of $\Sigma'$.  
For given $\Sigma$, it is always possible to subdivide to a smooth $\Sigma'$.

In particular,
$M_\Sigma = M_{\Sigma'}$, and 
$\LL(\Sigma)$ is a closed subset of $\LL(\Sigma')$, so there is a fully faithful inclusion 
$Sh_{\LL(\Sigma)}(\widehat{M}_\Sigma) \subset  Sh_{\LL(\Sigma')}(\widehat{M}_\Sigma)$.  In fact  \cite{Ku}  shows that
this inclusion is intertwined with the pullback $\pi^* : \Coh(\bT(\Sigma)) \to \Coh(\bT(\Sigma'))$.  

Now let $\sigma$ be a cone of $\Sigma$.  We write $\Phi := Nbd(\sigma)$; it is an open sub-fanifold of $\Sigma$ which contains
exactly one closed stratum (namely $\sigma$).  Then $\bT(\Phi)$ is the toric variety which is the 
closure of the orbit corresponding to $\sigma$; i.e., $\bT(\Phi) = \bT(\Sigma/\sigma)$.   We write $\Phi'$ for the 
same subset as $\Phi$ of $M_\Sigma \otimes \RR$, but with fanifold structure restricted from $\Sigma'$.  Then 
$\bT(\Phi')$ is the preimage of $\bT(\Phi)$ under the toric blowdown.  

Meanwhile, $\LL(\Phi)$ is naturally identified with an open subset of $\LL(\Sigma)$; in fact it is  
a product of a trivial factor with $\LL(\Sigma/\sigma).$  Likewise $\LL(\Phi')$ is naturally an open 
subset of $\LL(\Sigma')$.  Meanwhile $\LL(\Sigma)$ is a closed subset of $\LL(\Sigma')$ and 
correspondingly $\LL(\Phi)$ of $\LL(\Phi')$.  

For brevity we write $\mu sh(X)$ for $\Gamma(X, \mu sh_X)$.  Let us contemplate the diagram: 

\begin{equation} \label{eq: cube} 
\begin{tikzcd}
\mu sh(\LL(\Sigma))^{op} \arrow[dr, hookrightarrow] \arrow[rrrr, equal, "\mbox{\cite{Ku}}"] \arrow[ddd] & & & & \Coh( \bT(\Sigma)) \arrow[ddd]  \arrow[dl] \\ 
&  \mu sh(\LL(\Sigma'))^{op} \arrow[d] \arrow[rr,equal, "\mbox{\cite{Ku}}"] & & \Coh( \bT(\Sigma') )\arrow[d] & \\ 
& \mu sh(\LL(\Phi'))^{op}  \arrow[rr,equal, "\mbox{Thm. \ref{thm: local mirror symmetry}}"] & & \Coh(\bT(\Phi')) & \\ 
\mu sh(\LL(\Phi))^{op} \arrow[ur, hookrightarrow] \arrow[rrrr,equal, "\mbox{\cite{Ku}}"] & & & & \Coh( \bT(\Phi)) \arrow[ul] \\ 
\end{tikzcd}
\end{equation} 

Our task is to show that the outer square commutes (or more precisely to construct the natural transformation realizing the commutativity).  
It will suffice to show that the inner square and the four trapezoids commute, and that diagonal morphisms are all fully faithful. (Given commutativity, 
it is enough to show full faithfulness of the left diagonals). 

The morphisms on the right trapezoid are all pullback of coherent sheaves; it commutes.  The vertical morphisms of the left trapezoid are restriction
of microsheaves to open sets, and the diagonal morphisms are inclusions of the full subcategory of microsheaves supported on a closed subset of the
given microsupport; these obviously commute.  As we have already mentioned, commutativity of the upper trapezoid is established in \cite{Ku}.
Commutativity of the lower trapezoid follows from applying global sections over $\Phi$ to this result.  Finally, all fans in the central square 
are smooth, so its commutativity is Theorem \ref{thm: local mirror symmetry}. 
\end{proof} 

\begin{remark}
We can see from the proof that if $\Phi$ is any fanifold and $\Phi'$ is a fanifold obtained by subdividing its strata, then 
$fsh_\Phi$ is naturally a subsheaf of full subcategories of $fsh_{\Phi'}$. 
\end{remark}

We can also remove the assumption that the toric components $\bT(\Sigma)$ of the large complex structure limit variety are varieties rather than stacks, by generalizing slightly our understanding of what the data comprises a fan $\Sigma.$  There are various levels of generality of the notion of stacky fan: see \cite{stackyfans} for details. 
Kuwagaki's result \cite{Ku} is proven for the following class: 

\begin{definition}[\cite{stackyfans, Ku}]
  A {\em stacky fan} is the data of a map of lattices $\beta: \widetilde{M} \to M$ with finite cokernel, together with fans $\widetilde{\Sigma} \subset \widetilde{M} \otimes \RR$ and $\Sigma \subset M \otimes \RR$,
  such that $\beta$ induces a combinatorial equivalence on the fans.  
\end{definition}

As explained in \cite{stackyfans}, the usual GIT description of a toric variety from a fan extends in the obvious way to stacky fans, and for any cone $\sigma$ in the fan, the failure of the stacky generators of $\sigma$ to be primitive contributes an isotropy group to the corresponding stratum of the toric DM stack $\bT(\Sigma).$ 
One way to prescribe a stacky fan is to fix integral generators on the rays of an ordinary fan; these are then taken to be the images of the basis vectors of $\widetilde{M}$, and 
$\Sigma'$ is defined by lifting the cones of $\Sigma$ in the only possible way.

\begin{example} \label{ex: smooth stacky} Let $\Sigma\subset \RR^2$ be the fan whose nonzero cones are $\langle v_1\rangle, \langle v_2\rangle, \langle v_1,v_2\rangle,$ where we set $v_1 = (-1,1)$ and $v_2=(1,1),$ so that the usual toric variety associated to $\Sigma$ is the singular quadric $\{xy = z^2\}.$ 
Giving $\Sigma$ the structure of a stacky fan by fixing these generators remembers that the inclusion of lattices 
$\ZZ\langle v_1,v_2\rangle\hookrightarrow \RR\langle v_1,v_2\rangle \cap \ZZ$ has index 2, and the corresponding toric 
stack is $\CC^2/(\ZZ/2)$ (whose coarse moduli space is the singular quadric mentioned above).  
\end{example}

On the A-side, the definition of FLTZ Lagrangian $\LL(\Sigma)$ also generalizes in the obvious way to the case of stacky fans \cite{FLTZ3, Ku}, where now for a cone $\sigma,$ the failure of stacky generators to be primitive contributes a finite abelian group component to the corresponding torus 
$\sigma^\perp\subset \widehat{M},$ so that the torus $\sigma^\perp$ is no longer connected.
See \cite[Figures 9,12,13]{GS17} for images of stacky FLTZ Lagrangians.

The category $\mathrm{Fan}^{\twoheadrightarrow}$ admits an evident generalization  $\mathrm{StackyFan}^{\twoheadrightarrow}$, where now a morphism 
$(M, \widetilde{M}, \Sigma, \widetilde{\Sigma}) \to (M', \widetilde{M}', \Sigma', \widetilde{\Sigma}')$ is given by 
the choice of some cone $\widetilde{\sigma} \in \widetilde{\Sigma}$ whose image we denote $\sigma = \beta(\widetilde{\sigma})$, and compatible isomorphisms 
$$(M/\sigma, \widetilde{M}/\widetilde{\sigma}, \Sigma/\sigma, \widetilde{\Sigma}/\widetilde{\sigma}) \cong (M', \widetilde{M}', \Sigma', \widetilde{\Sigma}')$$
The appropriate notion of smooth stacky fan is that for which the corresponding toric stack is smooth as a stack; note as 
in Example \ref{ex: smooth stacky}, the underlying fan $\Sigma$ is simplicial but not necessarily smooth. 

We then define the notion of stacky fanifold  by changing Definition \ref{defn:fanifold} to require
 a map $\Exit(\Phi) \to \mathrm{StackyFan}^{\twoheadrightarrow}$.  (The comparison to normal cones still happens from $\Sigma \subset M \otimes \RR$.)

\vspace{4mm}

\section{Epilogue}  \label{sec: epilogue}

In this article we have established the homological mirror symmetry 
\[
\Fuk(\bW(\Phi)) = \Coh(\bT(\Phi))
\]
between the Fukaya category
of a certain noncompact symplectic manifold and the category of coherent sheaves on a certain singular algebraic space (or stack).  
We now outline the strategy to deform this result to a proof of mirror symmetry for smooth compact fibers of toric degenerations.  The
broad strokes of this strategy are well known to experts and have been implemented in some special cases 
\cite{Seidel-deformations, Seidel-genus2, Seidel-quartic, Sheridan-CY}; the key new point here is the r\^ole of Corollary \ref{cor: viterbo functoriality}. 

The first step is to understand how to construct a compact symplectic manifold $\overline{\bW}(\Phi)$ containing $\bW(\Phi)$ 
as the complement of a normal crossings divisor $D = D_1 \cup \ldots \cup D_n$.  We expect that such $D$ and $\overline{\bW}(\Phi)$ can be constructed 
by gluing together our local understanding of $\bW(\Phi)$ as a pair-of-pants complement.  (When $\Phi = \Sigma \cap S^n$, the existence of
such a smooth $\overline{\bW}(\Phi)$ follows from \cite{GS17, Zhou-skel}, although such a local gluing description of it does not.) 

By general principles, the Fukaya category of  $\overline{\bW}(\Phi)$ contains a deformation of the Fukaya category of $\bW(\Phi)$.  Indeed, for Lagrangians
disjoint from $D$, the essential difference between the definitions of these categories is that the former counts disks passing through $D$, and the latter does not.  
By SFT stretching, we may instead work entirely in $\bW(\Phi)$ and count disks asymptotic to certain Reeb orbits, and pair the result with the class
$\alpha \in \mathrm{SH}^\bullet(\bW(\Phi))[[Q_1, \ldots, Q_n]]$
which counts disks in a neighborhood of the divisor,  passing through the divisor.

The identification $\mathrm{SH}^\bullet(\bW(\Phi)) \cong
\mathrm{HH}^\bullet(\mathrm{Fuk}(\bW(\Phi)))$ matches this picture with the abstract deformation theory of categories, and so we may carry the class $\alpha$ across
homological mirror symmetry and ask whether the corresponding class in $\mathrm{HH}^\bullet(\mathrm{Coh}(\bT(\Phi)))[[Q_1, \ldots, Q_n]]$ 
arises from a deformation of $\bT(\Phi)$ to a smooth Calabi-Yau.
This is the key remaining point, and its resolution in existing works such as \cite{Seidel-deformations, Seidel-genus2, Seidel-quartic, Sheridan-CY} depends
on using special symmetries of the particular $\bT(\Phi)$ of interest there.  

By the  Hochschild-Kostant-Rosenberg theorem, one can pass from Hochschild cohomology to polyvector fields 
$\mathrm{H}^\bullet(\bT(\Phi), \Lambda^\bullet \mathbb{T}_{\bT(\Phi)})$ 
(using the appropriately derived version, where $\mathbb{T}$ denotes the tangent complex of $\bT(\Phi)$), and from this perspective, what must be shown is that the deformation class lives in
$\mathrm{H}^1(\bT(\Phi), \mathbb{T}_{\bT(\Phi)})$ and is a smoothing deformation there.  

Both questions are naturally studied locally on an affine cover.  Let us observe that our 
Corollary \ref{cor: viterbo functoriality} allows us to translate them into questions about $\alpha$ expressed locally in terms
of its Viterbo restrictions to a pair-of-pants cover of $\bW(\Phi)$.   
We will return to the construction of $\overline{\bW}(\Phi)$ and the aforementioned local study of the properties of $\alpha$ in a future work.

\bibliographystyle{plain}
\bibliography{refs}

\begin{thebibliography}{10}

\bibitem{AS}
Mohammed Abouzaid and Paul Seidel.
\newblock An open string analogue of {V}iterbo functoriality.
\newblock {\em Geometry \& Topology}, 14(2):627--718, 2010.

\bibitem{AGEN3}
Daniel \'{A}lvarez Gavela, Yakov Eliashberg, and David Nadler.
\newblock Arborealization {III}: {P}ositive arborealization of polarized
  {W}einstein manifolds.
\newblock {\em arXiv:2011.08962}.

\bibitem{Avdek}
Russell Avdek.
\newblock Liouville hypersurfaces and connect sum cobordisms.
\newblock {\em arXiv:1204.3145}, 2012.

\bibitem{Cieliebak-Eliashberg}
Kai Cieliebak and Yakov Eliashberg.
\newblock {\em From Stein to Weinstein and back: symplectic geometry of affine
  complex manifolds}, volume~59.
\newblock American Mathematical Soc., 2012.

\bibitem{Costello}
Kevin~J Costello.
\newblock The {G}romov-{W}itten potential associated to a {TCFT}.
\newblock {\em arXiv:math/0509264}, 2005.

\bibitem{Caldararu-Tu}
Andrei C\u{a}ld\u{a}raru and Junwu Tu.
\newblock Computing a categorical {G}romov-{W}itten invariant.
\newblock {\em Compos. Math.}, 156(7):1275--1309, 2020.

\bibitem{Eliashberg}
Yakov Eliashberg.
\newblock Weinstein manifolds revisited.
\newblock {\em arXiv:1707.03442}, 2017.

\bibitem{FLTZ2}
Bohan Fang, Chiu-Chu~Melissa Liu, David Treumann, and Eric Zaslow.
\newblock A categorification of {M}orelli's theorem.
\newblock {\em Inventiones mathematicae}, 186(1):79--114, Oct 2011.

\bibitem{FLTZ1}
Bohan Fang, Chiu-Chu~Melissa Liu, David Treumann, and Eric Zaslow.
\newblock T-duality and homological mirror symmetry for toric varieties.
\newblock {\em Advances in Mathematics}, 229(3):1873 -- 1911, 2012.

\bibitem{FLTZ3}
Bohan Fang, Chiu-Chu~Melissa Liu, David Treumann, and Eric Zaslow.
\newblock {The Coherent-Constructible Correspondence for Toric Deligne-Mumford
  Stacks}.
\newblock {\em International Mathematics Research Notices}, 2014(4):914--954,
  11 2012.

\bibitem{GR1}
Dennis Gaitsgory and Nick Rozenblyum.
\newblock {\em A study in derived algebraic geometry: Volume I: correspondences
  and duality}, volume 221.
\newblock American Mathematical Society, 2019.

\bibitem{GR2}
Dennis Gaitsgory and Nick Rozenblyum.
\newblock {\em A study in derived algebraic geometry: Volume II: deformations,
  {L}ie theory and formal geometry}, volume 221.
\newblock American Mathematical Society, 2020.

\bibitem{GS17}
Benjamin Gammage and Vivek Shende.
\newblock Mirror symmetry for very affine hypersurfaces.
\newblock {\em arXiv:1707.02959}.

\bibitem{GPS3}
Sheel Ganatra, John Pardon, and Vivek Shende.
\newblock Microlocal {M}orse theory of wrapped {Fukaya} categories.
\newblock {\em arXiv:{1809.08807}}.

\bibitem{GPS2}
Sheel Ganatra, John Pardon, and Vivek Shende.
\newblock Sectorial descent for wrapped {F}ukaya categories.
\newblock {\em arXiv:{1809.03472}}.

\bibitem{GPS1}
Sheel Ganatra, John Pardon, and Vivek Shende.
\newblock Covariantly functorial wrapped {F}loer theory on {L}iouville sectors.
\newblock {\em Publications math{\'e}matiques de l'IH{\'E}S}, 131(1):73--200,
  2020.

\bibitem{Ganatra-Perutz-Sheridan}
Sheel Ganatra, Timothy Perutz, and Nick Sheridan.
\newblock Mirror symmetry: from categories to curve counts.
\newblock {\em arXiv:1510.03839}, 2015.

\bibitem{stackyfans}
Anton Geraschenko and Matthew Satriano.
\newblock Toric stacks {I}: {T}he theory of stacky fans.
\newblock {\em Trans. Amer. Math. Soc.}, 367(2):1033--1071, 2015.

\bibitem{GS-intro}
Mark Gross and Bernd Siebert.
\newblock Affine manifolds, log structures, and mirror symmetry.
\newblock {\em Turk J Math}, 27:33--60, 2003.

\bibitem{GS-log1}
Mark Gross and Bernd Siebert.
\newblock Mirror symmetry via logarithmic degeneration data {I}.
\newblock {\em J. Differential Geom.}, 72(2):169--338, 02 2006.

\bibitem{GS-log2}
Mark Gross and Bernd Siebert.
\newblock Mirror symmetry via logarithmic degeneration data, {II}.
\newblock {\em J. Algebraic Geom.}, 19(4):679--780, 2010.

\bibitem{Kashiwara-Schapira}
Masaki Kashiwara and Pierre Schapira.
\newblock {\em Sheaves on Manifolds}, volume 292.
\newblock Springer Science \& Business Media, 2013.

\bibitem{KKP}
Ludmil Katzarkov, Maxim Kontsevich, and Tony Pantev.
\newblock Hodge theoretic aspects of mirror symmetry.
\newblock {\em arXiv:0806.0107}, 2008.

\bibitem{Kontsevich-ICM}
Maxim Kontsevich.
\newblock Homological algebra of mirror symmetry.
\newblock In {\em Proceedings of the international congress of mathematicians},
  pages 120--139. Springer, 1995.

\bibitem{Kontsevich-Soibelman}
Maxim Kontsevich and Yan Soibelman.
\newblock Homological mirror symmetry and torus fibrations.
\newblock {\em arXiv:math/0011041}, 2000.

\bibitem{Ku}
Tatsuki Kuwagaki.
\newblock The nonequivariant coherent-constructible correspondence for toric
  stacks.
\newblock {\em Duke Math. J.}, 169(11):2125--2197, 08 2020.

\bibitem{Lee}
Heather Lee.
\newblock Homological mirror symmetry for open {R}iemann surfaces from
  pair-of-pants decompositions.
\newblock {\em arXiv:1608.04473}, 2016.

\bibitem{Nadler-wrapped}
David Nadler.
\newblock Wrapped microlocal sheaves on pairs of pants.
\newblock {\em arXiv:1604.00114}, 2016.

\bibitem{NS20}
David Nadler and Vivek Shende.
\newblock Sheaf quantization in {W}einstein symplectic manifolds.
\newblock {\em arXiv:2007.10154}.

\bibitem{Schwede}
Karl Schwede.
\newblock Gluing schemes and a scheme without closed points.
\newblock In {\em Recent progress in arithmetic and algebraic geometry}, volume
  386 of {\em Contemp. Math.}, pages 157--172. Amer. Math. Soc., Providence,
  RI, 2005.

\bibitem{Seidel-deformations}
Paul Seidel.
\newblock Fukaya categories and deformations.
\newblock {\em arXiv:math/0206155}, 2002.

\bibitem{Seidel-genus2}
Paul Seidel.
\newblock Homological mirror symmetry for the genus two curve.
\newblock {\em arXiv:0812.1171}, 2008.

\bibitem{Seidel-speculations}
Paul Seidel.
\newblock Some speculations on pairs-of-pants decompositions and {F}ukaya
  categories.
\newblock {\em arXiv:1004.0906}, 2010.

\bibitem{Seidel-quartic}
Paul Seidel.
\newblock {\em Homological mirror symmetry for the quartic surface}, volume
  236.
\newblock American Mathematical Society, 2015.

\bibitem{shende-microlocal}
Vivek Shende.
\newblock Microlocal category for {W}einstein manifolds via h-principle.
\newblock {\em arXiv:1707.07663}, 2017.

\bibitem{shende-toric}
Vivek Shende.
\newblock Toric mirror symmetry revisited.
\newblock {\em arXiv:2103.05386}, 2021.

\bibitem{shende-takeda}
Vivek Shende and Alex Takeda.
\newblock {C}alabi-{Y}au structures on topological {F}ukaya categories.
\newblock {\em arXiv:1605.02721}, 2016.

\bibitem{Sheridan-CY}
Nick Sheridan.
\newblock Homological mirror symmetry for {C}alabi--{Y}au hypersurfaces in
  projective space.
\newblock {\em Inventiones mathematicae}, 199(1):1--186, 2015.

\bibitem{SYZ}
Andrew Strominger, Shing-Tung Yau, and Eric Zaslow.
\newblock Mirror symmetry is {T}-duality.
\newblock {\em Nuclear Physics B}, 479(1-2):243--259, 1996.

\bibitem{Weinstein}
Alan Weinstein.
\newblock Contact surgery and symplectic handlebodies.
\newblock {\em Hokkaido Mathematical Journal}, 20(2):241--251, 1991.

\bibitem{Zhou-skel}
Peng Zhou.
\newblock Lagrangian skeleta of hypersurfaces in $(\mathbb{C}^\times)^n$.
\newblock {\em Sel. Math. New Ser.}, 26(26), 2020.

\end{thebibliography}
\end{document}